\documentclass[11pt,oneside,reqno]{amsart} 
\usepackage[vmargin=3.3cm,a4paper]{geometry}
\usepackage[utf8]{inputenc}
\usepackage[T1]{fontenc}
\usepackage{amsmath,amsthm,amssymb}
\usepackage{mathtools}
\usepackage{graphicx}
\usepackage{xcolor}
\usepackage{enumitem}
\usepackage{marginnote}
\usepackage{hyperref}

\definecolor{Darkgreen}{rgb}{0,0.4,0}
\hypersetup{
  colorlinks,
  linkcolor=Darkgreen,
  citecolor=Darkgreen,
  breaklinks=true,
  bookmarksnumbered=true,
  plainpages=false,
  unicode=true,
  pdfauthor={Jiří Černý},
  pdftitle={Giant component for the supercritical level-set percolation of the Gaussian
  free field on regular expander graphs}
  }

\newtheorem{theorem}{Theorem}[section]

\newtheorem{lemma}[theorem]{Lemma}
\newtheorem{proposition}[theorem]{Proposition}

\theoremstyle{definition}

\newtheorem{assumptions}[theorem]{Assumption}

\theoremstyle{remark}
\newtheorem{remark}[theorem]{Remark}

\numberwithin{equation}{section}

\setlist[enumerate]{label=(\alph*)}

\DeclareMathOperator{\Cov}{Cov}

\DeclareMathOperator{\Var}{Var}

\DeclareMathOperator{\spn}{span}

\DeclareMathOperator{\children}{desc}
\DeclareMathOperator{\anc}{anc}

\DeclarePairedDelimiter\abs{\lvert}{\rvert}
\DeclarePairedDelimiter\norm{\lVert}{\rVert}

\newcommand{\refoverseteq}[1]{\overset{\mathclap{\eqref{#1}}}=\,}
\newcommand{\refoversetgeq}[1]{\overset{\mathclap{\eqref{#1}}}\ge\,}
\newcommand{\refoversetleq}[1]{\overset{\mathclap{\eqref{#1}}}\le\,}

\newcommand{\D}{\mathrm{d}}

\newcommand{\Max}{\mathrm{max}}
\newcommand{\Sec}{\mathrm{sec}}
\newcommand{\bbone}{\boldsymbol 1}
\newcommand{\Id}{\mathrm{Id}}

\newcommand{\finitegraph}{{\mathcal{G}_n}}
\newcommand{\vfinitegraph}{{\mathcal{V}_n}}
\newcommand{\efinitegraph}{{\mathcal{E}_n}}
\newcommand{\tfinitegraph}{{\tilde{\mathcal{G}}_n}}
\newcommand{\tvfinitegraph}{{\tilde{\mathcal{V}}_n}}
\newcommand{\bfinitegraph}{{\bar{\mathcal{G}}_n}}
\newcommand{\bvfinitegraph}{{\bar{\mathcal{V}}_n}}
\newcommand{\hfinitegraph}{{\hat{\mathcal{G}}_n}}
\newcommand{\hvfinitegraph}{{\hat{\mathcal{V}}_n}}
\newcommand{\treegraph}{{\mathbb T_d}}
\newcommand{\ttreegraph}{{\tilde{\mathbb T}_d}}

\newcommand{\rot}{{\textup{o}}}

\newcommand{\ack}{\subsection*{Acknowledgment}}

\begin{document}

\title[Giant component for the level-set percolation of the GFF on expanders]
{Giant component for the supercritical level-set percolation of the Gaussian
  free field on regular expander graphs}

\author[J. Černý]{Jiří Černý}
\address{Jiří Černý, Department of Mathematics and Computer Science,
  University of Basel, Spiegelgasse~1, 4051 Basel, Switzerland}
\email{jiri.cerny@unibas.ch}

\date{September 30, 2021}

\begin{abstract}
  We consider the zero-average Gaussian free field on a certain class of
  finite $d$-regular graphs for fixed $d\ge 3$. This class includes $d$-regular
  expanders of large girth and typical realisations of random $d$-regular
  graphs. We show that the level set of the zero-average Gaussian free
  field above level $h$ has a giant component in the whole supercritical
  phase, that is for all $h<h_\star$, with probability tending to one as
  the size of the graphs tends to infinity. In addition, we show that
  this component is unique. This significantly improves the result of
  \cite{AC20b}, where it was shown that a linear fraction of vertices is
  in mesoscopic components if $h<h_\star$,  and together with the
  description of the subcritical phase from \cite{AC20b} establishes a
  fully-fledged percolation phase transition for the model.
\end{abstract}

\keywords{Gaussian free field; level-set percolation; giant component;
  expander graphs}

\maketitle

\section{Introduction} 
\label{sec:intro}

Level-set percolation of the Gaussian free field is a significant
representative of percolation models with long range correlations. It has
attracted attention for a long time, dating back to
\cite{MS83,LS86,BLM87}. In the last decade, it has been subject to
intensive research after a non-trivial percolation phase transition has been identified
for this model on $\mathbb Z^d$ in
\cite{RS13}, see for instance \cite{AS18,CN20,DPR18,Szn19a}. Only very
recently, in the remarkable paper \cite{DGRS20}, it has been shown that
this phase transition is sharp, and, rather amazingly,  the critical
exponents have been identified for a related model of the level-set
percolation on the cable system of $\mathbb Z^d$ in \cite{DPR21}.

In coherence with a long line of past percolation results, it is
natural to consider the level-set percolation of the Gaussian free field on
finite graphs as well. In this context, \cite{Aba19} introduced
a suitable version of the Gaussian free field which can be defined on
finite graphs, the \emph{zero-average} Gaussian free field, and studied
its properties on discrete tori of
growing side length in dimension $d\ge 3$.
In \cite{AC20b} (with preparatory steps conducted in \cite{AC20a}),
A.\,Abächerli and the author initiated the
investigations of the zero-average Gaussian free field on a certain class
of finite locally tree-like $d$-regular graphs. The present paper
continues these investigations.

In \cite{AC20b} it has been shown that the
level-set percolation of the zero-average Gaussian field on this class of
graphs exhibits a percolation phase transition at a critical level
$h_\star$ in the following sense: With probability tending to one as the
size $N$ of the graphs tends to infinity, whenever $h > h_\star$, the
level set of the zero-average Gaussian free field above level $h$ does
not contain any connected component of size larger than $C_h \log N$,
and, on the contrary, whenever $h < h_\star$, a linear fraction of the
vertices is contained in `mesoscopic' connected components of the level
set above level $h$, that is in components having a size of at least a
small fractional power of $N$. The critical level $h_\star$ agrees with
the percolation threshold of the level set percolation of the usual
Gaussian free field on a $d$-regular tree which was identified in
\cite{Szn16}.

In the subcritical phase, $h>h_\star$, this description of the behaviour
of the level set is satisfactory. On the other hand, in the
supercritical phase, $h<h_\star$, it leaves open the question whether
the mesoscopic components form a giant component, that is a component of
size of order $N$, cf. \cite[Remark~5.7]{AC20b}.

This is a natural question since for other probabilistic models on
essentially the same class of graphs the emergence of the unique giant
component has been shown in the corresponding supercritical phases. The
first example is the Bernoulli bond percolation on $d$-regular expanders
of large girth considered in \cite{ABS04} (for more recent results, see
  \cite{KLS20}). A second example is the percolation of the vacant set
left by the simple random walk on the same class of graphs as considered
here and in \cite{AC20b}, see \cite{CTW11}. In particular the latter
result gives a strong indication that a giant component should emerge
also in the supercritical phase of the level-set percolation, as the two
models share many common features, like similar decay of correlations.

We answer this question affirmatively. To state our results we
first recall the setting of \cite{AC20b}. We fix $d\ge 3$ and assume
that $(\finitegraph)_{n\ge 1}$ is a sequence of graphs satisfying the
following conditions.
\begin{assumptions}
  \label{assumptions}
  There exist $\alpha \in (0,1)$, $\beta >0$, and an increasing sequence of
  positive integers $(N_n)_{n\ge 1}$ with
  $\lim_{n\to\infty}N_n =\infty$ such that for all $n\ge 1$:
  \begin{enumerate}
    \item $\finitegraph = (\vfinitegraph,\efinitegraph)$ is a simple
      connected graph with $N_n$ vertices which is $d$-regular (that is all
        its vertices have degree $d$).

    \item For all $x \in \vfinitegraph$ there is at most
      one cycle in the ball of radius $\lfloor \alpha \log_{d-1}(N_n) \rfloor$
      around $x$.

    \item The spectral gap  of $\finitegraph$, denoted by
      $\lambda_\finitegraph$, satisfies $\lambda_\finitegraph \ge \beta$.
  \end{enumerate}
\end{assumptions}

We refer to \cite{AC20b} for a more detailed discussion of these
assumptions, but recall that they are satisfied for two important classes
of graphs: (a) random $d$-regular  graphs, (b) $d$-regular expanders of
large girth. We also remark that assumptions very similar to ours were
used in recent studies of quantum ergodicity on graphs, and in related
studies of percolation of the level sets of the adjacency eigenvectors (see for
  instance \cite{ES10,AL15}).

On $\finitegraph$ we consider the zero-average Gaussian free field
$\Psi_\finitegraph = (\Psi_\finitegraph(x):{x\in \mathcal V_n})$ which is a
centred Gaussian process on $\mathcal V_n$ whose law
is determined by its covariance function
\begin{equation}
  E [\Psi_\finitegraph(x) \Psi_\finitegraph(y)] =
   G_\finitegraph(x,y) \qquad \text{for all $x,y \in \vfinitegraph$},
\end{equation}
where $G_\finitegraph(\cdot,\cdot)$ is the zero-average Green function
on
$\finitegraph$ (see \eqref{def:zmGreen}, \eqref{eqn:greenfcts} for its definition).

The zero-average Gaussian free field is a natural version of the Gaussian
free field for finite graphs. However, due to the zero-average property,
namely
\begin{equation}
  \label{eqn:zeroaverage}
  \sum_{x\in \vfinitegraph} \Psi_\finitegraph(x) = 0, \qquad \text{a.s.},
\end{equation}
it comes with some peculiarities like the lack of an
FKG-inequality and of the domain Markov property which are instrumental
when studying Gaussian free field on infinite graphs, cf.~\cite[Section~2.2]{AC20b}.

We analyse the properties of the level sets of
$\Psi_\finitegraph$ above level $h\in \mathbb R$, that is of
\begin{equation}
  \label{def:levelset}
  E^{\ge h}(\Psi_\finitegraph)
  \coloneqq \{x\in \vfinitegraph  :   \Psi_\finitegraph(x)\ge h\}.
\end{equation}
In particular, we are interested in the sizes of its largest and second
largest connected components $\mathcal C_{\Max}^{\finitegraph,h}$ and
$\mathcal C_{\Sec}^{\finitegraph,h}$.

For our investigations it is important that the field
$\Psi_{\finitegraph}$ is locally well approximated by the Gaussian free
field $\varphi_\treegraph = (\varphi_{\treegraph}(x): x\in \treegraph)$
on the infinite rooted $d$-regular tree  $\treegraph$ (see the paragraph
  containing~\eqref{eqn:defgfftd} for the definition). For now, we only
define its percolation function
\begin{equation}
  \label{eqn:defeta}
  \eta (h) \coloneqq P\big(\abs{\mathcal C_\rot^{h}} =
    \infty\big),
\end{equation}
where $\mathcal C_{\rot}^{h}$ is the connected component
of the set
$E^{\ge h}(\varphi_\treegraph)
\coloneqq \{x\in \treegraph : {\varphi_\treegraph(x)\ge h}\}$
containing the root
$\rot \in \treegraph$, and we set
\begin{equation}
  \label{def:hstar}
  h_\star \coloneqq \inf \big \{ h\in \mathbb R :  \eta (h) = 0 \big \},
\end{equation}
to be its critical value. From \cite{Szn16} it is known that
$h_\star$ is positive and finite.

We can now state our main result.

\begin{theorem}
  \label{thm:main}
  If $h<h_\star$, then for every sequence of graphs
  $(\finitegraph)_{n\ge 1}$ satisfying Assumption~\ref{assumptions} and
  every $\delta >0$
  \begin{equation}
    \label{eqn:main}
    \lim_{n\to\infty} P\Big(\frac{\abs{ \mathcal
          C_\Max^{\finitegraph,h}}}{N_n}\in (\eta (h)-\delta, \eta
        (h)+\delta  )
      \ \text{and} \ \abs[\big]{\mathcal
        C_\Sec^{\finitegraph,h}}\le \delta   N_n \Big) = 1.
  \end{equation}
\end{theorem}

Theorem~\ref{thm:main} confirms the emergence of the giant component in
the supercritical phase of the model, gives its typical size, and
provides its uniqueness. Together with the description of the subcritical
behaviour from Theorem~4.1 of \cite{AC20b} (which states that
  $\lim_{n\to\infty} P(\abs{\mathcal C_\Max^{\finitegraph,h}} \le C_{h}
    \log N_n) =1$
  for $h>h_\star$) it establishes a fully-fledged percolation phase
transition for the level-set percolation of the zero-average Gaussian free
field on~$\finitegraph$.

Assumption~\ref{assumptions} of Theorem~\ref{thm:main}
can be weakened slightly, as explained in Remark~\ref{rem:assumptions} at
the end of the paper. However, for these weakened assumptions we do not
have the corresponding subcritical description. We thus prefer to work
in the same setting as in~\cite{AC20b}.

Similarly as in \cite{ABS04,CTW11,KLS20}, we use a sprinkling technique
to show that the mesoscopic components (that we know to exist due to
  \cite[Theorem~5.1]{AC20b}) indeed form a giant component. Making the
sprinkling work in the settings of dependent percolation is however
rather challenging, as was already observed in \cite{CTW11}, in the
context of the vacant set left by a random walk. In the context of Gaussian free
field, sprinkling was previously used in \cite{DR15}, to show the
existence of an infinite connected component of the supercritical level set
on $\mathbb Z^d$ when $d\to \infty$. The diverging dimension is important
for the arguments therein, since the correlations of the field decay with
the dimension (as $d^{-1}$ for the neighbouring vertices). Several
sprinkling steps are also used in the recent paper
\cite{DGRS20}, which proves the sharpness of the phase transition for the
level set of the Gaussian free field on $\mathbb Z^d$, $d\ge 3$. Note
also that the results of \cite{DGRS20} can be combined with \cite{Aba19}
to show the existence of the giant component for the supercritical level
set of the zero-average Gaussian free field on a large discrete torus
(cf.~\cite[Section~1.2]{DGRS20}).

Very recently, a result similar to Theorem~\ref{thm:main} was proved by
G.\,Conchon-Kerjan in \cite{Con21}. Namely, assuming
that $\mathcal G_n$ is a \emph{uniformly random}
$d$-regular graph with $n$ vertices, he shows that under the \emph{annealed} probability
measure $P_{\mathrm{ann}}$ (that is taking into account the randomness of
  the graph and the field),
\begin{equation}
  \label{eqn:con}
  \lim_{n\to\infty }P_{\mathrm{ann}}\Big(
    \frac{\abs{\mathcal C_\Max^{\finitegraph,h}}}{n} \in (\eta (h)-\delta , \eta
      (h)+\delta )  \text{ and }
    \frac{ \abs{\mathcal C_\Sec^{\finitegraph,h}}}{\log n}\in(c,c')
    \Big) = 1,
\end{equation}
for every $\delta >0$ and some $0<c<c'<\infty$,
and further gives a
rather detailed description of the geometry of
$\mathcal C_\Max^{\finitegraph,h}$.
The arguments of \cite{Con21} are
completely different from the ones used in this paper and rely strongly
on the assumption that $\mathcal G_n$ is a random regular graph, and thus
can be revealed progressively using the usual pairing construction. As
discussed above, this assumption is stronger than our
Assumption~\ref{assumptions}.

The estimate on $\abs{\mathcal C_\Sec^{\finitegraph,h}}$ in
\eqref{eqn:con} is essentially optimal  and better than our estimate on
the same quantity in \eqref{eqn:main}. Incidentally, this resembles the
previously known results for the vacant set of random walk: On the same
class of graphs as here \cite[Theorem~1.3]{CTW11}, only shows that the
second largest connected component of the vacant set is $o(N_n)$, while
on random regular graphs it can be proved that it is $O(\log N_n)$, by
combining the techniques of \cite{CF13,MR98}. Note also that \cite{KLS20}
shows that for any $\omega <1$ there are regular expanders with an
arbitrarily large girth such that the second connected component of the
Bernoulli bond percolation has size growing at least as $N_n^\omega $,
which indicates that the exact asymptotic behaviour of
$\abs{\mathcal C_\Sec^{\finitegraph,h}}$ might be a delicate issue.

Let us now comment on the proof of Theorem~\ref{thm:main}. To explain
its main ideas, it is useful to discuss the sprinkling
construction for the Bernoulli percolation from \cite{ABS04} first. This
construction relies on the
fact that a percolation configuration
$(\omega^p(x))_{x\in \vfinitegraph} \in \{0,1\}^{\vfinitegraph}$ at level
$p$ can be obtained as the maximum of two independent Bernoulli
configurations $\omega^{p_1}$ and $\omega^{p_2}$ with the levels $p_1$, $p_2$
satisfying $1-p=(1-p_1)(1-p_2)$. For the techniques of \cite{ABS04} to work,
it is very important that (a) $\omega^{p_2}$ is
independent of~$\omega^{p_1}$, (b) $\omega^{p_2}$ is a Bernoulli percolation, that is
the random variables $(\omega^{p_2}(x))_{x\in \vfinitegraph}$ are
independent, and (c) that the maximum function is monotonous, in particular
$\{x:{\omega^{p}(x)=1}\}\supset \{{x:\omega^{p_1}(x)=1}\}$. While (c) is
important for the sprinkling not to destroy the mesoscopic components of
$\omega^{p_1}$, (a) and (b) play a key role in estimating the probability of a certain
bad event which needs to be much smaller than $e^{-cN_n/m_n}$,
with $m_n$ denoting the minimal size of mesoscopic components (cf.~proof
  of Proposition~3.1 in \cite{ABS04}). In \cite{ABS04}, the proof of this
estimate is just a simple large deviation argument for i.i.d.~Bernoulli
random variables. Unfortunately, a corresponding estimate is mostly
simply not true in the setting of correlated Gaussian fields.

Before describing our approach, let us very quickly mention two natural
ideas how to adapt the sprinkling construction of \cite{ABS04} to the
zero-average Gaussian free field which,  unfortunately, cannot
easily be converted into a rigorous proof, mostly due to the lack of
independence.  The first one is to use the existence of many mesoscopic
components at a level $h'\in (h,h_\star)$ and prove that by lowering the
level from $h'$ to $h$ those components merge. This preserves the
monotonicity, that is the point (c) from the last paragraph, but
completely destroys the independence (a) and (b), making the above
mentioned estimate on the bad event essentially impossible to prove. The
second one is to write $\Psi^\finitegraph$ as a linear combination
$ \sqrt {1-t^2}\Psi_\finitegraph' +  t \Psi_\finitegraph''$ (with a small
  $t$) of its independent copies $\Psi_\finitegraph', \Psi_\finitegraph''$.
Here, the monotonicity (c) is lost (but probably could be salvaged by
  some technical work), (a) is preserved, but the correlations of
$\Psi_\finitegraph''$ make the estimate on the bad event fail again.
Remark also that the zero average property \eqref{eqn:zeroaverage}
excludes writing $\Psi_\finitegraph $ as a sum $X+Y$ of two non-trivial
independent fields $X$, $Y$ such that $Y\ge 0$, or as $\max(X,Y)$ for $X$,
$Y$ independent; both of these decompositions would be desirable for the
monotonicity~(c).

In this paper we thus develop a new decomposition of the zero-average
Gaussian free field which provides enough independence to be useful in a
sprinkling argument and which is of independent interest,
see Section~\ref{sec:decomposition}. It is inspired by a similar
decomposition of the (usual) Gaussian
free field on $\mathbb Z^d$ from \cite{DGRS20}. Using this decomposition
we will  write $\Psi_\finitegraph$ as a sum of two independent components
$\Psi_\finitegraph = \Psi^1_\finitegraph + \Psi^2_\finitegraph$, where
\begin{equation}
  \label{eqn:introsprinkling}
  \Psi^2_\finitegraph(x) \coloneqq t_n\Big(Z_0(x) - N^{-1}\sum_{y\in \vfinitegraph}
      Z_0(y)\Big), \qquad x\in \mathcal V_n,
\end{equation}
for some family $(Z_0(x))_{x\in \vfinitegraph}$ of i.i.d.~Gaussian random
variables.  Since the field $\Psi^2_\finitegraph$ is essentially
an i.i.d.~field, up to the zero-average property, this writing preserves
(a) and (b) from the above discussion, but gives up on the monotonicity (c).
We will deal with the non-monotonicity issue by taking $t_n$ small and by
restricting the connected components of the level set to certain
subgraphs of $\finitegraph$ where $Z_0$ is not too small. These arguments
are relatively straightforward and
are given in Sections~\ref{sec:expansion},~\ref{sec:sprinkling}.

The decomposition, however, introduces a new problem: the field
$\Psi_\finitegraph^1$ is not longer a zero-average Gaussian free field
and we thus do not know that it has many mesoscopic components in the
whole supercritical phase $h<h_\star$. To show this we will use a
perturbative argument. More precisely, we use the fact that
$\Psi_\finitegraph^1 = \Psi_\finitegraph - \Psi^2_\finitegraph$ (with
  $\Psi_\finitegraph$ and $\Psi_\finitegraph^2$ dependent!) and that
$\Psi_\finitegraph$ has many mesoscopic components at any level
$h'\in (h,h_\star)$, by \cite[Theorem~5.1]{AC20b}. We then show that,
typically, these components are robust to certain perturbations and are
thus not destroyed by subtracting $\Psi^2_\finitegraph$. The proof of the
existence
of the robust components is based on multi-type branching process arguments
developed in \cite{AC20b}. Its details are given in
Sections~\ref{sec:robust}--\ref{sec:mesoscopic}. On the way, in
Section~\ref{sec:coupling}, we use the decomposition of
$\Psi_\finitegraph$ from Section~\ref{sec:decomposition} to construct a
new coupling of $\Psi_\finitegraph$ and~$\varphi_\treegraph$.

\section{Preliminaries}
\label{sec:prelim}

In this section we introduce the notation and recall few useful facts
that we use throughout the paper. For an arbitrary locally-finite,
simple, non-oriented graph $G$ we denote by $V(G)$ and $E(G)$ the sets of
its vertices and edges. For $x,y\in V(G)$, we write $x\sim y$ when
$(x,y)\in E(G)$, $d_G(\cdot,\cdot)$ denotes their graph distance, and
$\deg_G(x)$ the degree of $x$ in $G$. For any $U\subset  V(G)$,  $|U|$
stands for its cardinality, and
$\partial_G  U \coloneqq \{y\in V(G) \setminus U: \exists {x\in U} \text{ s.t. }
  {x\sim y} \}$
denotes its (outer vertex) boundary.  For any $r\ge0 $ and $x \in  V(G)$
we define the ball and sphere of radius $r$ around $x$ to be
$B_G(x,r) \coloneqq \{y \in V(G) : {d_G(x,y) \leq r} \} $ and
$S_G(x,r) \coloneqq \{y \in V(G) : d_G(x,y) = r \} $.

We write $\bar P_x^{G}$ for
the canonical law on $V(G)^{\mathbb N}$ of the \emph{lazy} simple random
walk $X=(X_k)_{k\ge 0}$ on $G$ starting at $x \in V(G)$, and $\bar E_x^{G}$ for the
corresponding expectation. Under $\bar P_x^G$, the transition probabilities of $X$
are given by
\begin{equation}
  \bar P_x^G(X_{k+1}=z\mid  X_k =y) = \begin{cases}
    \tfrac 12,&\text{if $z=y$},\\
    \tfrac 1{2\deg_G(x)},\quad& \text{if $(z,y)\in E(G)$}.
  \end{cases}
\end{equation}
If $G$ is a finite connected graph, we denote the unique invariant
distribution of $X$ by
\begin{equation}
  \label{eqn:pidef}
  \pi_G (x) \coloneqq \frac{\deg_G(x)}{2  \abs{E(G)}}.
\end{equation}
The zero-average Green function $\bar G_G$ of $X$ and its density are given by
\begin{equation}
  \label{def:zmGreen}
  \begin{split}
    \bar G_G(x,y) &\coloneqq  \sum_{k \ge 0} (\bar P^G_x(X_k = y) - \pi_G (y)),
    \qquad x,y\in V(G),
    \\ \bar g_G(x,y) &\coloneqq (\deg_G(y))^{-1} \bar G_G(x,y).
  \end{split}
\end{equation}
It is easy to check from the reversibility of the random walk that $\bar g_G(x,y)$ is a
symmetric function.
Zero-average Gaussian free field on $G$ is a
centred Gaussian process $(\Psi_G(x))_{x\in V(G)}$
whose law is determined by its covariance function
\begin{equation}
  \label{eqn:generalzerogff}
  E(\Psi_G(x) \Psi_G(y)) = C_0 \bar g_G(x,y) \qquad \text{for all } x,y\in
  V(G).
\end{equation}
The constant $C_0$ only influences the scaling of the field and is
introduced for convenience. If $G$ is $d$-regular, as in
Assumption~\ref{assumptions}(a), it customary to take
$C_0=d/2$. With this choice,
\begin{equation}
  \label{eqn:greenfcts}
  \bar g_G(x,y) = \tfrac 12 \bar G_G(x,y) = G_G(x,y),
\end{equation}
where $G_G(\cdot,\cdot)$ is the zero-average Green function of the
usual continuous-time random walk on $G$ (the factor $\frac 12$ disappears due
  to the laziness), and thus the covariance from \eqref{eqn:generalzerogff} agrees with
the one used in \cite{AC20b}, cf.~(2.17) therein.

For any field $f$ on $G$ we denote by
$E^{\ge h}(f) \coloneqq \{x\in V(G): f(x)\ge h\}$ its level set above
the level $h\in \mathbb R$.

We use $\treegraph$ to denote the $d$-regular infinite tree with root
$\rot$. For every $x\in V(\treegraph)$ we denote by $\children(x)$ the
set of its direct descendants, and for
$x\in V(\treegraph) \setminus \{\rot\}$ we use $\anc(x)$ to denote the
direct ancestor of $x$ in $\treegraph$. The Gaussian free field on
$\treegraph$ is a centred Gaussian process
$(\varphi_\treegraph(x))_{x\in V(\treegraph)}$
whose distribution is determined by
\begin{equation}
  \label{eqn:defgfftd}
   E(\varphi_\treegraph(x)\varphi_{\treegraph}(y)) =
  g_{\treegraph}(x,y) \qquad \text{for all } x,y\in V(\treegraph),
\end{equation}
where $g_\treegraph$ is the Green function of the (usual discrete-time)
simple random walk on $\treegraph$.

As mentioned earlier, we consider for fixed $d\ge 3$ the $d$-regular
graphs $(\finitegraph)_{n\ge 1}$ satisfying Assumption~\ref{assumptions},
and we abbreviate $\vfinitegraph = V(\finitegraph)$,
$\efinitegraph = E(\finitegraph)$. For $r\ge 0$, we say that a vertex
$x\in \vfinitegraph$ is $r$-treelike, if there is no cycle in
$B_\finitegraph(x,r)$. If $x$ is $r$-treelike, then we fix a graph
isomorphism $\rho_{x,r}:B_\finitegraph(x,r) \to B_\treegraph(\rot, r)$
such that $\rho_{x,r}(x) = \rot$.

We recall from \cite[Proposition~2.2]{AC20b} that there is
$\varepsilon \in (0,1)$ such that for every $n\ge 1$ and
$x,y\in \vfinitegraph$,
\begin{equation}
  \label{eqn:greenfunctionest}
  \bar g_\finitegraph(x,y) \le C (d-1)^{-d_\finitegraph(x,y)} +
  N_n^{-\varepsilon }.
\end{equation}

Finally, Assumption~\ref{assumptions}(a,c) imply (by Cheeger's
  inequality, for the argument see e.g.~\cite[(2.11)]{CTW11}) the uniform
isoperimetric inequality for the sequence $\finitegraph$:
\begin{equation}
  \label{eqn:expansion}
  \parbox{7.0cm}{There is $\beta '>0$ such that
    $\frac{\abs{\partial_\finitegraph A}}{\abs{A}}\ge \beta'$ for all
    $n\ge1$ and $A\subset \vfinitegraph$ with
    $\abs A \le \abs \vfinitegraph/2$.}
\end{equation}

We use $c$, $c'$, $C$, \dots to denote
positive constants with values changing from place to place and which only depend on
the degree $d$ and the constants $\alpha $ and $\beta $ from
Assumption~\ref{assumptions}.

\section{Decomposition of the field} 
\label{sec:decomposition}

The goal of this section is to construct a decomposition of the
zero-average Gaussian free field into independent components.  We believe
that this decomposition is of independent interest. It is the main
ingredient of our sprinkling construction, as described in the
introduction, but also will be used in Section~\ref{sec:coupling} to
construct a new coupling of $\Psi_\finitegraph$ and $\varphi_\treegraph$.
The construction of this decomposition is inspired by a similar
decomposition for the usual Gaussian free field on $\mathbb Z^d$ from
\cite{DGRS20}, see Lemma~3.1 therein. However, the zero-average property
introduces certain complications making the decomposition less
straightforward.

For sake of generality, we consider an arbitrary finite, simple,
non-oriented, connected graph $G=(V(G), E(G))$ in this section. That is we
do not require that Assumption~\ref{assumptions} holds.

Recall the definition of $\Psi_G$ from \eqref{eqn:generalzerogff}.
To introduce its decomposition we need more notation. We write $\tilde G$
for the graph obtained from $G$ by adding an additional vertex to the middle of every
edge of $G$, formally $\tilde G = (V(\tilde G), E(\tilde G))$ with
\begin{align}
  V(\tilde G) &\coloneqq V(G) \cup E(G),\\
  E(\tilde G) &\coloneqq \{(x,e): x \in V(G), e\in E(G), e \ni x\}.
\end{align}
Observe that $\tilde G$ is a bipartite graph.
For $\tilde x\in V(\tilde G)$, let
\begin{equation}
  \label{eqn:tildepi}
  \tilde \pi_G(\tilde x)
  \coloneqq \deg_{\tilde G}(\tilde x)
  = \begin{cases}
    \deg_G(\tilde x),\quad&\text{if }\tilde x \in V(G),\\
    2,\quad&\text{if }\tilde x \in E(G).
  \end{cases}
\end{equation}
Let
\begin{equation}
  \tilde Q_G(\tilde x,\tilde y) =
  \bbone_{(\tilde x,\tilde y) \in  E(\tilde G)}/\tilde\pi_G (\tilde x),
  \qquad \tilde x,\tilde y\in V(\tilde G),
\end{equation}
be the transition matrix of the \emph{usual} simple
random walk on $\tilde G$. $\tilde Q_G$ acts on the space
$\ell^2(\tilde \pi_G )$ by
$\tilde Q_G f (\tilde x)
= \sum_{\tilde y\in  V(\tilde G)} \tilde Q_G(\tilde x,\tilde y)f(\tilde y)$.
Due to the reversibility, $\tilde Q_G$ is a self-adjoint operator on
$\ell^2(\tilde\pi_G )$. Since $\tilde G$ is connected and bipartite,
$\tilde Q_G$ has simple eigenvalues $1$ and $-1$ with respective eigenfunctions
$\bbone $ and~$w$, where $w(\tilde x) = 1$ if $\tilde x\in V(G)$ and
$w(\tilde x)=- 1$ if $\tilde x \in E(G)$. Denoting by
$\norm{\cdot}_{\tilde \pi_G }$ the  $\ell^2(\tilde \pi_G )$-norm and by
$\langle \cdot,\cdot \rangle_{\tilde \pi_G }$ the corresponding scalar
product, we have
\begin{equation}
  \label{eqn:normvw}
  \norm{\bbone}^2_{\tilde\pi_G } = \norm{w}^2_{\tilde\pi_G }
  = \sum_{\tilde x \in V(\tilde G)} \tilde\pi_G (\tilde x) = 4  \abs{E(G)}, \qquad
  \langle \bbone , w\rangle_{\tilde\pi_G } = 0.
\end{equation}

Let $\Pi_G $ be the orthogonal projection (in $\ell^2 (\tilde\pi_G )$) on
$\spn(\bbone, w)$. Then $\Pi_G^2 = \Pi_G$ and $\Pi_G$ is self-adjoint in
$\ell^2(\tilde\pi_G )$. Moreover, since $\bbone $ and $w$ are eigenvectors
of $\tilde Q_G$, the operators $\Pi_G$ and $\tilde Q_G$ commute, and thus, for
every $k\in \mathbb N$,
the operators
$\tilde Q_G^k \Pi_G = \Pi_G \tilde Q_G^k$ and $(\Id-\Pi_G)\tilde Q_G^k$ are self-adjoint
as well (here $\Id$ stands for the identity operator). For later use we
observe that for
$f\in \ell^2(\tilde \pi_G)$ and $y\in V(G)$ (so that $w(y) = 1$),
\begin{equation}
  \label{eqn:PIexpl}
  \begin{split}
    (\Pi_G f) (y)
    &= \norm \bbone_{\tilde \pi_G}^{-2} \langle f, \bbone  \rangle_{\tilde \pi_G}
    \bbone(y)
    + \norm w_{\tilde \pi_G}^{-2} \langle f, w  \rangle_{\tilde \pi_G}
    w(y)
    \\&\refoverseteq{eqn:normvw}
    \frac 1{4\abs {E(G)}}
    \sum_{\tilde x\in V(\tilde G)} f(\tilde x) (1+w(\tilde x)) \tilde \pi_G
    (\tilde x)
    \\&\refoverseteq{eqn:tildepi}
    \frac 1{2\abs {E(G)}}
    \sum_{x\in V(G)} f(x) \deg_{G}(x) \refoverseteq{eqn:pidef}
    \sum_{x\in V(G)} f(x)\pi_G(x).
  \end{split}
\end{equation}

With $C_0$ as in \eqref{eqn:generalzerogff}, let
$(Z_k (\tilde x))_{k \in \mathbb N,\tilde x \in V(\tilde G)}$ be
independent centred Gaussian random variables with
\begin{equation}
  \label{eqn:genZvariance}
  \Var Z_k (\tilde x) =  {C_0}/{\tilde\pi_G (\tilde x)},
\end{equation}
defined on
a probability space $(\Omega , \mathcal A, P)$. Finally, for
$x\in V(G)$, set
\begin{align}
  \xi_G^k( x)& \coloneqq \sum_{\tilde y \in V(\tilde G)}
  ((\Id - \Pi_G ) \tilde Q_G^k)(x, \tilde y) Z_k(\tilde y)
  \overset{\text{(not.)}}
  = \big((\Id-\Pi_G)\tilde Q_G^k Z_{k}\big)(x),
  \\
  \label{eqn:tildePsi}
  \tilde \Psi_G(x) &\coloneqq \sum_{k\in \mathbb N} \xi_G^k (x).
\end{align}

We now show that \eqref{eqn:tildePsi} provides the desired decomposition of
$\Psi_G$.

\begin{proposition}
  \label{pro:decomposition}
  The series on the right-hand side of \eqref{eqn:tildePsi} converges in
  $L^2(P)$ and $P$-a.s., and the
  law of  $\tilde \Psi_G$ agrees with the law of $\Psi_G$, that is
  $\tilde\Psi_G$ is a zero-average Gaussian free field on $G$.
\end{proposition}

\begin{proof}
  We start by computing the covariances of the fields $\xi_G^k$.
  Using the independence of $Z_k(\tilde x)$'s,
  the self-adjointness of $(\Id - \Pi_G)\tilde Q_G^k$ and the fact that $\Pi_G $
  and $\tilde Q_G$ commute, for every $x,y\in V(G)$,
  \begin{equation}
    \label{eqn:xicovariance}
    \begin{split}
      \Cov(\xi_G^k( x)&, \xi_G^k( y)) =\!\!
      \sum_{\tilde z\in V(\tilde G) }
      ((\Id - \Pi_G)\tilde Q_G^k)(x,\tilde z)
      ((\Id - \Pi_G)\tilde Q_G^k)(y,\tilde z) \frac {C_0} {\tilde\pi_G (\tilde z)}
      \\& =
      \frac {C_0}{\tilde\pi_G ( y)}
      \sum_{\tilde z\in V(\tilde G) }
      ((\Id - \Pi_G)\tilde Q_G^k)(x,\tilde z)
      ((\Id - \Pi_G)\tilde Q_G^k)(\tilde z,y)
      \\&=
      \frac {C_0}{\tilde\pi_G ( y)}
      ((\Id -\Pi_G)\tilde Q_G^{2k})( x,  y).
    \end{split}
  \end{equation}
  To compute the terms involving $\Pi_G$,
  let $(v_i)_{i=1,\dots, \abs{V(\tilde G)}}$ be an orthonormal
  basis of $\ell^2(\tilde\pi_G)$ composed by the eigenvectors of
  $\tilde Q_G$
  such that $v_1 = \bbone /\norm{\bbone}_{\tilde\pi_G }$ and
  $v_2 = w/\norm{w}_{\tilde\pi_G }$, and let
  $(\lambda_i)_{i=1,\dots,\abs{V(\tilde G)}}$ be the corresponding
  eigenvalues. Observe $\Pi_G v_i = 0$ for $i\ge 3$ and that
  $\tilde Q_G^{2k}(f) =
  \sum_{i=1}^{\abs{V(\tilde G)}}
  \langle v_i, f \rangle_{\tilde\pi_G } \lambda_i^{2k}v_i$. Hence, for
  $x,y\in V(G)$,
  \begin{equation}
    \label{eqn:PiQk}
    \begin{split}
      (\Pi_G \tilde Q_G^{2k})(x,y)
      &= \frac 1 {\tilde\pi_G (x) }
      \langle \bbone_x, (\Pi_G \tilde Q_G^{2k}) \bbone_y\rangle_{\tilde\pi_G }
      \\&=
      \frac 1{\tilde\pi_G(x)} \sum_{i=1}^{|V(\tilde G)|}
      \big\langle \bbone_x, \langle v_i, \bbone_y\rangle_{\tilde\pi_G }
      \lambda_i^{2k} \Pi_G v_i\big\rangle_{\tilde\pi_G }
      \\&=
      \langle \bbone_y,v_1\rangle_{\tilde\pi_G } v_1(x)
      + \langle \bbone_y,v_2\rangle_{\tilde\pi_G } v_2(x)
      \\&=
      (\Pi_G\bbone_y)(x)
      \refoverseteq{eqn:PIexpl}\pi_G (y).
    \end{split}
  \end{equation}
  Hence, by \eqref{eqn:xicovariance}, \eqref{eqn:PiQk} and \eqref{eqn:tildepi},
  \begin{equation}
    \label{eqn:covxi}
    \Cov (\xi_G^k( x), \xi_G^k( y))  = \frac{C_0}{\deg_G(y)}(\tilde
      Q_G^{2k}(x,y) - \pi_G (y)),
    \qquad x,y\in V(G), k\in \mathbb N.
  \end{equation}
  The matrix $\tilde Q_G^{2k}$ restricted to $V(G)$ agrees with the
  $k$-step transition matrix of the lazy random walk on $G$, that is
  $Q_G^{2k}(x,y)= \bar P^G_x(X_k=y)$. In
  particular, due to standard convergence results for Markov chains,
  $ \abs{Q_G^{2k}(x,y) - \pi_G (y)}\le C e^{-c k}$ for all $x,y\in V(G)$, and
  thus also  $\abs{\Cov (\xi_G^k( x), \xi_G^k( y))} \le Ce^{-ck}$. This
  implies that the series in \eqref{eqn:tildePsi} converges in $L^2(P)$.
  The a.s.~convergence is then standard, e.g. using Kolmogorov's maximal
  inequality. Finally, \eqref{eqn:covxi} implies that
  \begin{equation}
    \begin{split}
      \Cov\big( \tilde \Psi_G (x), \tilde \Psi_G(y)\big)
      &=
      \sum_{k\in \mathbb N} \Cov \big(\xi_G^k(x), \xi_G^k(y)\big)
      \\&= \frac {C_0} {\deg_G(y)}
      \sum_{k\in \mathbb N} \big(\bar P^G_x(X_k = y) - \pi_G (y)\big),
    \end{split}
  \end{equation}
  which agrees with the covariance of $\Psi_G$ from
  \eqref{eqn:generalzerogff}. Since $\tilde \Psi_G$ is obviously a centred
  Gaussian field, this completes the proof.
\end{proof}

\section{Coupling with a tree} 
\label{sec:coupling}

We now come back to our original setting of Assumption~\ref{assumptions}
and construct, in Proposition~\ref{pro:coupling} below, a coupling
between the zero-average Gaussian free field $\Psi_{\finitegraph}$  and
the Gaussian free field $\varphi_{\treegraph}$. A similar coupling is
provided by Theorem~3.1 of \cite{AC20b}. However, our
Proposition~\ref{pro:coupling} has several advantages: First, it has a
much simpler proof which is based on the decomposition from
Section~\ref{sec:decomposition}. Second, in its proof we also couple
the underlying $Z$-fields (cf.~Remark~\ref{rem:Zcoupling} below) which will
be important later. And third, in contrast to \cite{AC20b}, we use two
independent fields $\varphi_\treegraph$, $\varphi'_\treegraph$ in its
statement; this will simplify the application of the coupling in the
second moment computation in the proof of
Proposition~\ref{pro:mesoscopic} below.

For the statement recall from Section~\ref{sec:prelim} that $\rho_{x,r}$
denotes a fixed isomorphism of $B_\finitegraph(x,r)$ and $B_\treegraph(\rot,r)$,
if $x\in \vfinitegraph$ is $r$-treelike.

\begin{proposition}
  \label{pro:coupling}
  There are $c,C\in (0,\infty)$ such that for all $n,r\in \mathbb N$,
  and for all $x,x'\in \vfinitegraph$
  which are $2r$-treelike and satisfy
  $B_\finitegraph(x,2r) \cap B_\finitegraph(x',2r) = \emptyset$ there
  exists a coupling $\mathbb Q_n^{x,x'}$ of $\Psi_\finitegraph$ and two
  independent Gaussian free fields $\varphi_\treegraph$,
  $\varphi '_\treegraph$ such that for all $\varepsilon >0$
  \begin{equation}
    \begin{split}
      \label{eqn:coupling}
      \mathbb Q_n^{x,x'} \big[ &\max\{D(x,r),D(x',r)\}  >
        \varepsilon \big] \\
      &\leq C d(d-1)^r\Big(\exp\Big(-\frac{\varepsilon^2 e^{cr}}{18}\Big)+
        \exp\Big(-\frac{\varepsilon^2 N_n}{9(r+1)}\Big)\Big),
    \end{split}
  \end{equation}
  where
  \begin{equation}
    D(x,r) \coloneqq
          \max_{y \in B_\finitegraph(x,r) }
          \abs[\big]{ \Psi_\finitegraph(y) -
            \varphi_\treegraph(\rho_{x,2r}(y))}.
  \end{equation}
\end{proposition}

\begin{proof}
  We use the decomposition of $\Psi_\finitegraph$ from
  Section~\ref{sec:decomposition} and a corresponding decomposition of
  $\varphi_{\treegraph}$. Using the notation of
  Section~\ref{sec:decomposition}, let
  $ \tvfinitegraph\coloneqq V(\tfinitegraph)$, and let
  $Z = (Z_k(\tilde x))_{k\in \mathbb N, x\in \tvfinitegraph}$ be a
  collection of independent Gaussian random variables on some probability
  space $(\Omega ,\mathcal A, \mathbb Q_n^{x,x'})$ with variances
  (cf.~\eqref{eqn:genZvariance}, \eqref{eqn:tildepi}, we take $C_0 = d/2$
    as explained below~\eqref{eqn:generalzerogff})
  \begin{equation}
    \label{eqn:Zvariances}
    \Var Z_k(\tilde x) = \begin{cases}
      \frac 12,\qquad& \text{if } \tilde x\in \mathcal V_n,\\
      \frac d4,\qquad& \text{if } \tilde x\in \mathcal E_n.
    \end{cases}
  \end{equation}
  Set
  $\xi_\finitegraph^k(x) \coloneqq \big((\Id - \Pi_\finitegraph) \tilde
    Q_\finitegraph^k Z_k\big)(x)$ and
  \begin{equation}
    \label{eqn:represenationPsi}
    \Psi_{\mathcal G_n}(x) \coloneqq \sum_{k\ge 0} \xi_\finitegraph^k(x) =
    \sum_{k \ge 0}\big((\Id - \Pi_\finitegraph) \tilde Q_\finitegraph^k Z_k\big)(x).
  \end{equation}
  By Proposition~\ref{pro:decomposition}, $\Psi_{\finitegraph}$ has the law of
  zero-average Gaussian free field.

  We now introduce an analogous decomposition for the field $\varphi_\treegraph$,
  similarly to \cite[Lemma~3.1]{DGRS20}.
  Let $\ttreegraph$ be a graph obtained from
  $\mathbb T_d$ by adding a vertex to the middle of every edge, and let
  $\mathtt Z =(\mathtt Z_k(\tilde x))_{k\in \mathbb N,\tilde x \in V(\ttreegraph)}$
  be a collection of independent Gaussian random
  variables on the same probability space
  $(\Omega , \mathcal A, \mathbb Q_n^{x,x'})$ such that (cf.~\eqref{eqn:Zvariances})
  \begin{equation}
    \label{eqn:ttZvariances}
    \Var \mathtt Z_k(\tilde x) = \begin{cases}
      \frac 12,\qquad& \text{if } \tilde x\in  V(\treegraph),\\
      \frac d4,\qquad& \text{if } \tilde x\in  V(\ttreegraph)\setminus
      V(\treegraph).
    \end{cases}
  \end{equation}
  Denoting $\tilde{\mathtt Q}$ the transition matrix of the usual simple
  random walk on $\tilde{\mathbb T}_d$, we set
  $\zeta^k (x) \coloneqq (\tilde{\mathtt Q}^k\mathtt Z_k)(x)$ and
  \begin{equation}
    \label{eqn:represenationphi}
    \varphi_\treegraph(x) \coloneqq \sum_{k\ge 0} \zeta^k(x) =
    \sum_{k \ge 0}( \tilde{\mathtt Q}^k \mathtt Z_k)(x).
  \end{equation}
  Then $\varphi_\treegraph$ is a Gaussian free field on $\treegraph$.
  This can be shown by a straightforward adaptation of the proof for the
  Gaussian free field on $\mathbb Z^d$ from \cite{DGRS20} (or by adapting
    the proof of Proposition~\ref{pro:decomposition}, leaving out all
    terms involving the projection
    $\Pi_G$). By introducing an independent copy
  $\mathtt Z ' = (\mathtt Z'_k(x))_{k\in \mathbb N,x\in V(\ttreegraph)}$
  of $\mathtt Z$, we further define the field $\varphi '_\treegraph$ by
  a formula analogous to \eqref{eqn:represenationphi}, with $\mathtt Z'$ instead of
  $\mathtt Z$.

  Let
  $\tilde \rho_{x,2r}:B_{\tfinitegraph}(x,4r) \to B_\ttreegraph(\rot,4r)$
  be the natural extension of the isomorphism $\rho_{x,2r}$ to the balls
  in graphs $\tfinitegraph$ and $\ttreegraph$. (Note that the ball
    $B_\tfinitegraph(x,4r)$ is related to $B_\finitegraph(x,2r)$, since
    in $\tfinitegraph$ there are additional vertices in the middle of
    every edge of $\finitegraph$.) We now require that under
  $\mathbb Q_n^{x,x'}$ the underlying fields $Z$, $\mathtt Z$, and
  $\mathtt Z'$ satisfy the following equalities while otherwise being
  independent:
  \begin{equation}
    \begin{split}
      \label{eqn:Zequalities}
      \mathtt Z_k (\tilde \rho_{x,2r}(\tilde y)) &= Z_k(\tilde y)
      \qquad \text{for every }k\le 2r, \tilde y\in B_{\tfinitegraph}(x,4r),\\
      \mathtt Z'_k (\tilde \rho_{x',2r}(\tilde y)) &= Z_k(\tilde y)
      \qquad \text{for every }k\le 2r, \tilde y\in B_{\tfinitegraph}(x',4r).
    \end{split}
  \end{equation}
  Observe that this can be done without changing the distribution of $Z$,
  $\mathtt Z$ and $\mathtt Z'$, in particular the assumption
  $B_\finitegraph(x,2r)\cap B_\finitegraph(x',2r)=\emptyset$ is necessary
  for the independence of $\mathtt Z$ and $\mathtt Z'$. The assumption
  that $x$ is $2r$-treelike implies that the law of the image by
  $\tilde\rho_{x,2r}$ of the random walk on $\tfinitegraph$ started in
  $\tilde y\in B_\tfinitegraph(x,2r)$ and stopped on exiting
  $B_{\tfinitegraph}(x,4r)$ agrees with the law of the random walk on
  $\ttreegraph$ started in $\tilde\rho_{x,2r}(\tilde y)$ and stopped on
  exiting $B_{\ttreegraph}(\rot,4r)$, and that this random walk makes at
  least $2r$ steps before being stopped. As consequence the corresponding
  transition probabilities agree in the sense of
  \begin{equation}
    \label{eqn:sametransition}
    \begin{split}
      &\tilde Q_\finitegraph^k(\tilde y, \tilde y')
      = \tilde {\mathtt Q}^k (\tilde \rho_{x,2r}(\tilde y),\tilde
        \rho_{x,2r}(\tilde
          y'))
      \\&\text{for } \tilde y\in B_{\tfinitegraph}(x,2r),
      \tilde y' \in B_{\tfinitegraph}(x,4r),
      k\le 2r.
    \end{split}
  \end{equation}
  From \eqref{eqn:represenationPsi} and
  \eqref{eqn:represenationphi}--\eqref{eqn:sametransition}
  it follows that for every $y\in  B_\finitegraph(x,r)$
  \begin{equation}
    \label{eqn:coupling_approx}
    \begin{split}
      \Psi_\finitegraph(y) -{}& \varphi_\treegraph(\rho_{x,2r}(y))
      \\&=\sum_{k>  2r } \xi_\finitegraph^k(y)
      - \sum_{0\le k \le 2r} (\Pi_\finitegraph  \tilde Q_\finitegraph^k Z_k)(y)
      - \sum_{k>2r} \zeta^k (\rho_{x,2r}(y)),
    \end{split}
  \end{equation}
  and a similar equality holds when $x$ is replaced by $x'$ and
  $\zeta^k $ by $\zeta^{\prime k} \coloneqq \tilde {\mathtt Q}^k \mathtt Z_k'$.

  We now estimate the three sums on the right-hand side of
  \eqref{eqn:coupling_approx}. For the last one, we claim that there is a
  constant $c >0$ such that for every $\varepsilon >0 $, $k_0\ge 1$, and
  $y\in V(\treegraph)$,
  \begin{equation}
    \label{eqn:suma}
    \mathbb Q_n^{x,x'}\Big(
      \abs[\Big]{\sum_{k> k_0} \zeta^k(y)} \ge \frac \varepsilon3\Big)
    \le 2  \exp \Big(-\frac {\varepsilon^2 e^{ck_0}}{18}\Big) .
  \end{equation}
  Indeed, observe that $\zeta^k(y)$, $k\ge 0$, are independent Gaussian
  random variables with
  $\Var \zeta^k(y)   =  \tilde {\mathtt Q}^{2k} (y,y)/2$ (which can be
    proved by a similar
    computation as in \eqref{eqn:xicovariance}, recalling $C_0=d/2$).
  Therefore,
  \begin{equation}
    \Var \Big(\sum_{k> k_0}\zeta^k(y)\Big)
    = \frac 12 \sum_{k> k_0} \tilde {\mathtt Q}^{2k} (y,y)
    \le  e^{-c k_0},
  \end{equation}
  where we used the fact that the lazy random walk $(X_k)_{k\ge 0}$ on $\treegraph$ satisfies
  $\tilde {\mathtt Q}^{2k} (y,y)\le e^{-c k}$, which can easily be
  proved by observing that $d_\treegraph(\rot, X_n)$ is a random walk on
  $\mathbb N$ with a drift pointing away from $0$.
  Claim \eqref{eqn:suma} then follows by the usual Gaussian tail
  estimates.

  We proceed similarly for the first sum in \eqref{eqn:coupling_approx}.
  Using \eqref{eqn:covxi} and the standard
  estimate on the convergence to stationarity for finite Markov chains
  (see e.g.~\cite[(12.11), p.155]{LPW09}),
  \begin{equation}
    \label{eqn:mchconv}
    \Var \xi_\finitegraph^k(x)
    \refoverseteq{eqn:covxi} \frac 12 \Big(\tilde
      Q_\finitegraph^{2k}(x,x)-\frac 1 {N_n}\Big) \le e^{-
      \lambda_{\finitegraph}k},
  \end{equation}
  where $\lambda_\finitegraph \ge \beta $ is the spectral gap appearing in
  Assumption~\ref{assumptions}(c) (due to the laziness, there is no $2$ in
    the exponent). Therefore,
  for every $\varepsilon >0$, $k_0\ge 1$, and $y\in \vfinitegraph$,
  \begin{equation}
    \label{eqn:sumb}
    \mathbb Q_n^{x,x'}\Big(
      \abs[\Big]{\sum_{k>  k_0} \xi_\finitegraph^k(y)} \ge \frac \varepsilon 3 \Big)
    \le 2 \exp \Big(-\frac { \varepsilon^2 e^{\beta k_0}}{18}\Big) .
  \end{equation}

  Finally, for the second sum in \eqref{eqn:coupling_approx}, we claim that
  for every $\varepsilon >0$, $k_0\ge 1$ and $y\in \vfinitegraph$,
  \begin{equation}
    \label{eqn:sumc}
    \mathbb Q_n^{x,x'}\Big(
      \abs[\Big]{\sum_{0\le k\le  k_0}
        \sum_{\tilde z\in \tvfinitegraph}
        (\Pi_\finitegraph \tilde Q_\finitegraph^k)(y,\tilde z) Z_k(\tilde z) }
      \ge \frac \varepsilon 3 \Big)
    \le 2  \exp \Big(-  \frac{\varepsilon^2 N_n} {9 (k_0+1)}\Big).
  \end{equation}
  Indeed, by the same computation as in \eqref{eqn:xicovariance}--\eqref{eqn:PiQk},
  \begin{equation}
    \Var \Big( \sum_{0\le k\le k_0}
      \sum_{\tilde z\in \tvfinitegraph}
      (\Pi_\finitegraph \tilde Q_\finitegraph^k)(y,\tilde z) Z_k(\tilde z)\Big)
    = \tfrac 12 (k_0+1) \pi_\finitegraph (y) = \tfrac {k_0+1} {2 N_n},
  \end{equation}
  this follows by the same reasoning as above.

  Claim \eqref{eqn:coupling} then follows from
  \eqref{eqn:coupling_approx}, \eqref{eqn:suma}, \eqref{eqn:sumb}, and
  \eqref{eqn:sumc}
  using the triangle inequality, a union bound, and the fact that
  $\abs{B_{\treegraph}(\rot,r)} = \abs{B_{\finitegraph}(x,r)}
  =\frac{d(d-1)^r-2}{d-2} \le d(d-1)^r$,
  if $x$ is $2r$-treelike.
\end{proof}

\begin{remark}
  \label{rem:Zcoupling}
  Later, it will play the key role that the coupling $\mathbb Q_n^{x,x'}$ also couples
  the underlying $Z$-fields. In particular, we will use that
  $\mathbb Q_n^{x,x'}$-a.s.
  \begin{equation}
    \label{eqn:Z0coupling}
    \begin{split}
      Z_0(y) &= \mathtt Z_0(\rho_{x,2r}(y) ),
      \qquad y\in B_\finitegraph(x,2r),
      \\
      Z_0(y) &= \mathtt Z'_0(\rho_{x',2r}(y) ),
      \qquad y\in B_\finitegraph(x',2r) .
    \end{split}
  \end{equation}
  which follows directly from \eqref{eqn:Zequalities}.
\end{remark}

\section{Robust components of the GFF on the tree} 
\label{sec:robust}

In what follows, we assume that $\Psi_{\finitegraph}$,
$\varphi_\treegraph$, and  the underlying fields $Z$, $\mathtt Z$
are constructed on some probability space $(\Omega , \mathcal A, P)$ and
\eqref{eqn:represenationPsi}, \eqref{eqn:represenationphi} hold. As
explained in the introduction (cf.~\eqref{eqn:introsprinkling}), in the
sprinkling construction we will write $\Psi_\finitegraph$ as a sum of two
independent fields $\Psi^1_\finitegraph$ and $ \Psi^2_\finitegraph$. To
this end, let $t\in (0,1)$ be a parameter which will later depend on $n$, and
write  $Z_0= \sqrt {1-t^2} Z_0^1 + t Z_0^2$, where
$Z_0^i= (Z_0^i(x):x\in \tvfinitegraph)$, $i\in \{1,2\}$,
are two independent copies of
$Z_0$. Similarly as above
\eqref{eqn:represenationPsi}, we define
\begin{equation}
  \xi_\finitegraph^{0,i}(x):= (\Id - \Pi_\finitegraph)Z_0^i(x)
  = Z_0^i(x) -\frac 1 {N_n} \sum_{y\in
    \vfinitegraph} Z^i_0(y), \quad x\in \vfinitegraph, i\in\{1,2\},
\end{equation}
(where in the second equality we used \eqref{eqn:PIexpl} and
  $\pi_\finitegraph(x) = 1/N_n$),  and set
\begin{equation}
  \label{eqn:Psionetwo}
  \Psi^1_\finitegraph (x) \coloneqq \sqrt{1-t^2} \xi_\finitegraph^{0,1}(x)
  + \sum_{k\ge 1}\xi_\finitegraph^k (x)
    \qquad\text{and}\qquad
    \Psi^2_\finitegraph (x) \coloneqq t \xi_\finitegraph^{0,2}(x) .
\end{equation}
Then $\Psi_\finitegraph=\Psi^1_\finitegraph + \Psi^2_\finitegraph$, and
$\Psi^1_\finitegraph$,  $\Psi^2_\finitegraph$ are independent.

Next, we introduce two independent copies $\mathtt Z_0^1$,
$\mathtt Z_0^2$ of  $\mathtt Z_0$ so that
$\mathtt Z_0 = \sqrt {1-t^2}\mathtt Z^1_0+ t\mathtt Z^2_0$, and define
(cf.~\eqref{eqn:represenationphi})
\begin{equation}
  \label{eqn:phionetwo}
  \varphi^1_\treegraph (x) \coloneqq
  \sqrt{1-t^2} \mathtt Z_0^1(x) + \sum_{k\ge 1}\zeta^k(x)
  \qquad\text{and}\qquad
  \varphi^2_\treegraph (x) \coloneqq  t \mathtt Z^2_0(x).
\end{equation}
Then $\varphi_\treegraph = \varphi^1_\treegraph + \varphi^2_\treegraph$
and the summands are independent.

The goal of the next three sections is to show that the supercritical
level sets of $\varphi^1_{\treegraph}$, and as consequence also of
$\Psi^1_\finitegraph$, have large connected components. Unfortunately,  we
cannot apply the results of \cite{AC20a, AC20b} directly, because
$\Psi^1_\finitegraph$ and $\varphi^1_\treegraph$ are no longer Gaussian
free fields.  In this section, we thus show that for any $h<h_\star$ the
level set $E^{\ge h}(\varphi_{\treegraph})$ of the unmodified field
$\varphi_\treegraph$ has infinite components which are robust to certain
perturbations. In the next section, we use this result to show that the
level set $E^{\ge h}(\varphi^1_\treegraph)$ percolates if $h<h_\star$ and
$t$ is small enough.
Finally, in Section~\ref{sec:mesoscopic}, we transfer these result to the
field $\Psi^1_\finitegraph$, using the coupling from
Section~\ref{sec:coupling}.

For the sprinkling construction of Section~\ref{sec:sprinkling}, we need
to consider two types of perturbations of $E^{\ge h}(\varphi_\treegraph)$.
The first one comes from the field $\varphi^2_\treegraph$, as already
explained, and the second one from an independent Bernoulli percolation.
For the latter, let $\iota = (\iota (x))_{x\in V(\treegraph)}$ be
i.i.d.~Bernoulli random variables with $P({\iota (x) =1}) =p$ which are
independent of $\mathtt Z$ and thus of $\varphi_\treegraph$. The
robustness against the perturbation by $\varphi_\treegraph^2$ involves
certain averaging property for $\varphi_{\treegraph}$ and is driven by a
parameter $\gamma  \in [-\infty,0]$. Formally, for $x\in V(\treegraph)$
(recalling that $\children(x)$ is the set of direct descendants of $x$
  in~$\treegraph$), let $\mathcal K(h,p,\gamma )$ be the set of
\emph{robust} vertices in $E^{\ge h}(\varphi_\treegraph)$ defined by
\begin{equation}
  \label{eqn:robust}
  \mathcal K(h,p,\gamma) \coloneqq \Big\{x\in \vfinitegraph:
    \varphi_{\treegraph}(x) \ge h, \iota (x) = 1, \text{ and }
    \sum_{y\in \children(x)} \varphi_{\treegraph}(y) \ge \gamma \Big\},
\end{equation}
and let $\mathcal C_\rot^{h,p,\gamma }$ be the connected component of
$\mathcal K(h,p,\gamma )$ containing the root $\rot$. Note
that if $p=1$ and $\gamma = -\infty$, then $\mathcal C_\rot^{h,p,\gamma }$
agrees with the connected component $\mathcal C^{h}_\rot$ of the level set
$E^{\ge h}(\varphi_\treegraph)$ containing the root $\rot$. We set
\begin{align}
  \label{eqn:etahpg}
  \eta(h,p,\gamma ) &\coloneqq P (\abs {\mathcal C_\rot^{h,p,\gamma } } =
    \infty ),
  \\
  \label{eqn:calS}
  \mathcal S &\coloneqq \{(h,p,\gamma )\in \mathbb R\times[0,1]\times [-\infty,0]:
    \eta (h,p,\gamma ) >0\}.
\end{align}

The main result of this section is the following proposition which shows
that, in the supercritical regime, $\mathcal C_\rot^{h,p,\gamma }$ has
similar properties as $\mathcal C_\rot^{h}$, cf.~\cite[Theorems~5.1,
  5.3]{AC20a} or \cite[(2.14), (2.16)]{AC20b}.

\begin{proposition}
  \begin{enumerate}
    \item
      If $(h,p,\gamma )\in \mathcal S$  and  $h'<h$, $p'>p$,
      $\gamma'<\gamma$,  then also $(h',p',\gamma ')\in \mathcal S$. Moreover,
      for every $h<h_\star$ there is $p<1$ and $\gamma >-\infty$ such that
      $(h,p,\gamma )\in \mathcal S$.

    \item For every $(h,p,\gamma )$ in the interior $\mathcal S^0$ of
      $\mathcal S$ there is $\lambda_h^{p,\gamma }>1$ such that
      \begin{equation}
        \lim_{k\to \infty}
        P\big(\abs{\mathcal C_\rot^{h,p,\gamma }\cap
            S_\treegraph(\rot,k)}
          \ge (\lambda_h^{p,\gamma })^k/k^2\big) = \eta(h,p,\gamma ) > 0.
      \end{equation}

    \item
    The functions $(h,p,\gamma ) \mapsto \lambda_h^{p,\gamma }$ and
    $(h,p,\gamma )\mapsto \eta (h,p,\gamma )$ are continuous on~
    $\mathcal S^0$ (this includes the continuity at points
      $(h,p,-\infty)\in \mathcal S^0$).
  \end{enumerate}
  \label{pro:robust}
\end{proposition}

\begin{remark}
  We expect that $\mathcal S$ is open, that is $\mathcal S^0 = \mathcal S$. Proving this would require
  to study the critical behaviour of $\mathcal C_\rot^{h,p,\gamma }$ which
  goes beyond the scope of this paper.
\end{remark}

\begin{proof}[Proof of Proposition~\ref{pro:robust}]
  The proof uses multi-type branching process techniques and is a
  modification of the arguments given in Section~3 of \cite{Szn16} and in
  Sections~\hbox{3--5} of \cite{AC20a}. Here, we only explain how these
  arguments should be adapted to our setting and leave out the parts that
  are relatively standard in the context of the multi-type branching
  processes.

  We first recall the recursive construction of $\varphi_{\treegraph}$
  from \cite[Section 2.1]{AC20a}. Define  random variables
  \begin{equation}
    \label{eqn:recconstfirt}
    \begin{split}
      Y_\rot &\coloneqq \varphi_\treegraph (\rot)
      \\Y_x &\coloneqq \varphi_\treegraph(x)-\frac 1 {d-1}
      \varphi_\treegraph (\anc( x)),
      \qquad\text{for }x \in V(\treegraph) \setminus \{ \rot \}.
    \end{split}
  \end{equation}
  Then, by the domain Markov property of  $\varphi_\treegraph$,
  see \cite[(2.6),(2.7)]{AC20a}, $(Y_x)_{x\in V(\treegraph)}$ are
  independent random variables such that
  $Y_\rot \sim   \mathcal N(0,\frac{d-1}{d-2})$ and
  $Y_x\sim  \mathcal N(0,  \frac{d}{d-1})$ for $x\neq \rot$.

  The
  definition \eqref{eqn:recconstfirt} can be written as
  $\varphi_\treegraph(\rot) = Y_\rot$ and
  \begin{equation}
    \label{eqn:recconst}
    \varphi_\treegraph(x) = \frac 1{d-1} \varphi_\treegraph(\anc(x)) + Y_x
    \quad\text{for }x \in V(\treegraph) \setminus \{ \rot \}.
  \end{equation}
  The field $\varphi_\treegraph$ is thus determined by $Y_x$'s, by
  applying \eqref{eqn:recconst} recursively. This also implies
  that $\varphi_{\treegraph}$ can be viewed as a multi-type branching
  process. Indeed, we can view every $x\in S_\treegraph(\rot,k)$ as an
  individual in the $k$-th generation of the branching process with an
  attached type $\varphi_\treegraph(x)\in \mathbb R$.
  \eqref{eqn:recconst} can be then rephrased as: every individual $x$ has
  $d-1$ children ($d$ children if $x=\rot$) whose types, conditionally on
  $\varphi_\treegraph(x)$, are chosen independently according to the
  normal distribution
  $\mathcal N(\frac 1{d-1}\varphi_\treegraph(x),\frac{d}{d-1})$. This
  point of view can easily be adapted to $\mathcal C_\rot^{h}$,
  namely by considering the same multi-type branching process but
  instantly killing all individuals whose type does not exceed $h$.
  Relying on this point of view, \cite{AC20a} investigates the properties of
  $\mathcal C_\rot^h$ using branching process techniques.

  We now modify this construction to  apply to
  $\mathcal C_\rot^{h,p,\gamma }$. In addition to instantly killing the
  individuals whose type does not exceed $h$, we also kill individuals $x$
  for which $\iota (x)=0$, and we kill all direct descendants of $x$ if
  $\sum_{y\in \children(x)} \varphi_{\treegraph}(y) < \gamma $. Then
  the surviving individuals form a component
  $\bar {\mathcal  C}_\rot^{h,p,\gamma }$ which is slightly larger than
  $\mathcal C_\rot^{h,p,\gamma}$. More precisely, since we only kill the
  direct descendants of non-robust vertices, and not those vertices
  themselves,
  \begin{equation}
    \bar {\mathcal C}_\rot^{h,p,\gamma } = \mathcal C_\rot^{h,p,\gamma} \cup
    \Big\{x\in \partial_\treegraph \mathcal C_\rot^{h,p,\gamma} :
      \varphi_\treegraph(x) \ge h, \iota(x) =1, \sum_{y\in
        \children (x)}\varphi_{\treegraph}(y)\overset {(!)}< \gamma \Big\}.
  \end{equation}
  As consequence, $\abs{\bar {\mathcal C}_\rot^{h,p,\gamma }} = \infty$
  iff $\abs{{\mathcal C}_\rot^{h,p,\gamma }} = \infty$, and
  $\abs{\bar {\mathcal C}_\rot^{h,p,\gamma } \cap S_\treegraph(\rot, k)} \ge a$
  implies that
  $\abs{\mathcal C_\rot^{h,p,\gamma } \cap S_\treegraph(\rot, k-1)} \ge a/(d-1)$.
  Hence, it is sufficient to show claims (b,c) for
  $\bar {\mathcal C}_\rot^{h,p,\gamma}$ instead of
  $\mathcal C_\rot^{h,p,\gamma }$ (with an additional constant $(d-1)$).
  The advantage of the former is that it can be interpreted as a
  multi-type branching process.

  The key role in the investigations of \cite{AC20a} plays certain operator
  introduced in \cite{Szn16} in order to give a  spectral characterisation
  of the critical value $h_\star$. This operator is defined as follows, cf.
  \cite[Section 2.2]{AC20a}: Let $\nu $ be the centred Gaussian measure
  of variance $\frac{d-1}{d-2}$. For $h\in \mathbb R$,
  $f\in L^2(\nu )$ and $a\in \mathbb R$, set
  \begin{equation}
    \label{eqn:defLh}
      (L_h f) (a) \coloneqq (d-1) \bbone_{[h,\infty)}(a) \,
      E^Y \big[ (f \bbone_{[h,\infty)})(\tfrac{a}{d-1} + Y )
            \big],
  \end{equation}
  where $Y \sim \mathcal N(0,\tfrac{d}{d-1})$ and $E^Y$ is the
  expectation with respect to $Y$. The operator $L_h$ is the `mean value'
  operator corresponding to $\mathcal C_\rot^h$ when it is viewed as a multi-type
  branching process, more precisely, for any $x\neq \rot$ and $a\ge h$,
  \begin{equation}
    (L_h f) (a) =  E\Big[\sum_{y \in \mathcal C_{\rot}^{h}
        \cap \children(x)} f(\varphi_\treegraph(y)) \,\Big |\,
      \varphi_\treegraph(x) = a, x\in \mathcal C_\rot^h \Big].
  \end{equation}
  Denoting $\lambda_h$ the largest eigenvalue of $L_h$, the critical point
  $h_\star$ is given as the unique solution of the equation
  $\lambda_h = 1$, see \cite[Proposition~3.3]{Szn16}

  For $\bar {\mathcal  C}_{\rot}^{h,p,\gamma }$, the corresponding
  operator has a similar, slightly more complicated, form: For
  $f\in L^2(\nu )$ and $a\in \mathbb R$ (and for $x\neq \rot$,
    $a\ge h$ in the formula on the right-hand side of the first line,
    which is only included to motivate the definition),
  \begin{equation}
    \label{eqn:defLhpg}
    \begin{split}
      (L_h^{p,\gamma } &f) (a)
      =  E \bigg[\sum_{x \in \bar{\mathcal
            C}_{\rot}^{h,p,\gamma }
          \cap\children (x)} f(\varphi_\treegraph(y)) \, \bigg |\,
        \varphi_\treegraph(x)=a, x \in \bar{\mathcal C}_\rot^{h,p,\gamma }\bigg]
      \\& \coloneqq p\bbone_{[h,\infty)}(a)
      E^Y \bigg[
        \bbone_{[\gamma ,\infty)}\Big(\sum_{i=1}^{d-1}
          (\tfrac a{d-1} + Y_i)\Big)
        \sum_{i=1}^{d-1}
        (f\bbone_{[h,\infty)}) (\tfrac{a}{d-1} + Y_i ) \bigg],
      \\& = p(d-1)\bbone_{[h,\infty)}(a)
      E^Y \bigg[
        \bbone_{[\gamma ,\infty)}\Big(\sum_{i=1}^{d-1}
          (\tfrac a{d-1} + Y_i)\Big)
        (f\bbone_{[h,\infty)}) (\tfrac{a}{d-1} + Y_1 ) \bigg],
    \end{split}
  \end{equation}
  where $(Y_i)_{i=1,\dots,d-1}$ are i.i.d.~$\mathcal N(0,\tfrac{d}{d-1})$
  and $E^Y$ is the corresponding expectation. Note that
  $L_h = L_h^{1,-\infty}$.

  Contrary to $L_h$, the operator $L_h^{p,\gamma }$ is not self-adjoint
  in $L^2(\nu )$. We thus need an additional argument to show that
  (cf.~\cite[Proposition~3.1]{Szn16}):
  \begin{equation}
    \label{eqn:eigenexistence}
    \parbox{13cm}{The value
      $\lambda_h^{p,\gamma }\coloneqq \norm{L_h^{p,\gamma}}_{L^2(\nu )} =
      \sup\{\langle  g, L_h^{p,\gamma } g \rangle_{L^2(\nu )}:
        \norm{g}_{L^2(\nu )}=1\}$ is a simple eigenvalue of
      $L_h^{p,\gamma }$. Moreover, there is a unique, non-negative eigenfunction
      $\chi_h^{p,\gamma }\in L^2(\nu )$ of $L_h^{p,\gamma }$ corresponding
      to $\lambda_h^{p,\gamma }$ with
      $\norm{\chi_h^{p,\gamma }}_{L^2(\nu )}=1$.
      }
  \end{equation}

  To show this, observe first that from \eqref{eqn:defLh},
  \eqref{eqn:defLhpg} it follows that there exist functions
  $K_h, K_h^{p,\gamma }: [h,\infty)^2 \to (0,\infty)$ such that, for $a\ge h$ (which is the relevant
    range since $L_h f(a) =L_h^{p,\gamma } f(a) = 0$ for $a<h$),
  \begin{equation}
    \label{eqn:kernels}
    \begin{split}
      (L_h f)(a) &= \int_{[h,\infty)}  K_h(a,y) f(y) \nu (\D y), \quad
      \\
      (L_h^{p,\gamma } f)(a) &= \int_{[h,\infty)} K_h^{p,\gamma }(a,y) f(y)
      \nu (\D y).
    \end{split}
  \end{equation}
  Moreover, $K_h^{p,\gamma }\le K_h$ for all admissible values of $h$, $p$,
  and $\gamma $. Since $L_h$ is a Hilbert-Schmidt operator  on
  $L^2(\nu )$ (see \cite[(3.16)]{Szn16}), it follows that
  $L_h^{p,\gamma }$ is a Hilbert-Schmidt and thus compact operator on
  $L^2(\nu )$ as
  well.  By Riesz-Schauder theorem (see e.g.~\cite[Theorem~6.15]{RS80}),
  every $\lambda \neq 0$ in the spectrum $\sigma (L_h^{p,\gamma })$ of
  $L_h^{p,\gamma }$ is an eigenvalue of $L_h^{p,\gamma}$ and $0$ is the
  only possible limit point of $\sigma (L_{h}^{p,\gamma })$.
  Since
  $\lambda_h^{p,\gamma } = \norm{L_h^{p,\gamma }}_{L^2(\nu )}
  = \sup\{\abs{\lambda }: \lambda  \in \sigma (L_h^{p,\gamma })\}$,
  it follows that there
  is $\lambda \in \mathbb C $ with
  $\abs{\lambda } = \lambda_{h}^{p,\gamma }$ such that
  $L_h^{p,\gamma } \chi= \lambda  \chi $ for some (possibly complex
    valued) $\chi \in L^2(\nu )$ with $\norm{\chi }_{L^2(\nu )}=1$.

  Next, we verify that $\lambda > 0$ and $\chi \ge 0$. If $\chi$
  is not of the form $\chi = \beta g $ for some $\beta \in \mathbb C$
  and a real-valued non-negative function $g$, then
  $\norm{\abs\chi}_{L^2(\nu )}=1$ and
  $\langle \abs\chi, L_h^{p,\gamma } \abs\chi\rangle_{L^2(\nu )} > \lambda_h^{p,\gamma }$
  which leads to contradiction with the definition of
  $\lambda_h^{p,\gamma }$ in \eqref{eqn:eigenexistence}. Hence, $\chi = \beta g$. Since
  the multiplication by scalars preserves eigenfunctions, we can assume that
  $\beta  =1$, that is $\chi =g $ is non-negative as required.
  The equality $L_h^{p,\gamma } \chi = \lambda \chi$ then implies that
  $\lambda >0$ as well, and thus $\lambda  = \lambda_h^{p,\gamma }$.

  To finish the proof of \eqref{eqn:eigenexistence}, it remains to show
  that $\lambda_h^{p,\gamma } $ is a simple eingenvalue. We proceed
  similarly to \cite{Szn16}: If $f\in L^2(\nu )$ is an eigenfunction of
  $L_h^{p,\gamma }$ attached to $\lambda_h^{p,\gamma } $, then we can
  assume that it is non-negative, as explained in the last paragraph.
  Thus $\langle f-\alpha \chi,\bbone \rangle_{L^2(\nu)}=0$ for some
  $\alpha \ge 0$. The function $f-\alpha\chi$ is also an eigenfunction
  attached to $\lambda_h^{p,\gamma } $, so it is a multiple of a
  non-negative function. It follows that $f-\alpha \chi =0$ in $L^2(\nu )$.
  That is $\lambda_h^{p,\gamma }$ is a simple eigenvalue corresponding to
  $\chi = \chi_{h}^{p,\gamma }$, completing the proof of
  \eqref{eqn:eigenexistence}.

  Using \eqref{eqn:kernels}, one easily shows that  $\chi_h^{p,\gamma }(a)>0$ for
  $a\in [h,\infty)$. From \eqref{eqn:defLhpg}, \eqref{eqn:eigenexistence} it follows that
  $\lambda_h^{p,\gamma }$ is decreasing in $h$ and $\gamma $, and
  stricly increasing in $p$. Strict monotonicity in $h$ can be proved as
  in \cite[(3.23)]{Szn16}.
  Since $\chi_{h}^{p,\gamma }>0$ on $[h,\infty)$, for
  $\gamma > \gamma '$ we have
  \begin{equation}
    \begin{split}
      \lambda_h^{p,\gamma }
      &= \langle \chi_{h}^{p,\gamma },
      L_h^{p,\gamma } \chi_h^{p,\gamma }\rangle_{L^2(\nu )}
      \overset{\eqref{eqn:defLhpg}} <
      \langle \chi_{h}^{p,\gamma },
      L_h^{p,\gamma' } \chi_h^{p,\gamma }\rangle_{L^2(\nu )}
      \\&
      \refoversetleq{eqn:eigenexistence}
      \langle \chi_{h}^{p,\gamma' },
      L_h^{p,\gamma' } \chi_h^{p,\gamma' }\rangle_{L^2(\nu )} =
      \lambda_h^{p,\gamma '},
    \end{split}
  \end{equation}
  yielding the strict monotonicity in $\gamma $.
  The continuity of $\lambda_h^{p,h}$ can be shown using the same
  arguments as in the proof of \cite[(3.20)]{Szn16}. In particular, the
  continuity at ${\gamma =-\infty}$ follows from the lower-semicontinuity
  of $\lambda_{h}^{p,\gamma }$ (cf.~\eqref{eqn:eigenexistence}) and its
  monotonicity.

  The rest of the proof of Proposition~\ref{pro:robust} follows the lines
  of \cite{Szn16,AC20a} with mostly obvious modifications, frequently
  relying on the fact that $L_{h}^{p,\gamma }$ is ``smaller'' than $L_h$
  (in the sense explained under \eqref{eqn:kernels}): First, as in
  \cite[Proposition~3.3]{Szn16}, it can be shown that the value of
  $\lambda_h^{p,\gamma }$ dictates whether the
  process is sub- or supercritical,
  \begin{equation}
    \label{eqn:Sinc}
    \{(h,p,\gamma ): {\lambda_h^{p,\gamma }\ge 1}\}\supset
    \mathcal S \supset \mathcal S^0
    = \{(h,p,\gamma ): {\lambda_h^{p,\gamma }>1}\},
  \end{equation}
  where $\mathcal S$ is as in \eqref{eqn:calS}, and the equality
  in \eqref{eqn:Sinc} follows
  from the strict monotonicities of $\lambda_h^{p,\gamma }$ discussed in the
  last paragraph. Second, the same argument as in the proof of
  \cite[Proposition~3.1(i)]{AC20a} provides a control on the growth of
  $\chi_h^{p,\gamma }$, which is necessary for the further steps. Third,
  Section~4 of~\cite{AC20a} (studying an functional equation for the
    non-percolation probability) needs to be adapted: besides changing
  the defininition of the non-linear operator $R_h$ from
  \cite[(4.3),(4.4)]{AC20a}
  accordingly, only relatively straightforward changes are required there.

  After these preparatory steps, the continuity of $\eta $ in claim
  (c) of the proposition can be proved in the same way as Theorem~5.1,
  and  claim (b) in the same way as Theorem~5.3 in \cite{AC20a}. The
  first part of claim (a) follows by monotonicity. Finally,
  using the continuity of
  $\lambda_h^{p,\gamma }$ from (c)
  \begin{equation}
    \label{eqn:lambdaconv}
    \lim_{\gamma \to -\infty }\lambda_{h}^{p,\gamma }
    \overset{\eqref{eqn:defLhpg}}=
    p  \lim_{\gamma \to -\infty }\lambda_{h}^{1,\gamma } = p \lambda_h.
  \end{equation}
  Hence if $h<h_\star$ and thus $\lambda_h >1$, then there exist $p\in (0,1)$
  and $\gamma \in \mathbb R$ with $\lambda_h^{p,\gamma }>1$, proving the second
  part of (a).
\end{proof}

\section{Percolation for the pruned field on the tree}
\label{sec:pruned}

We now consider the \emph{pruned} field
$\varphi_\treegraph^1$ defined in \eqref{eqn:phionetwo} (recall that
  $\varphi_\treegraph^1$ implicitly depends on the sprinkling
  strength $t$) and show that for $h<h_\star$ and $t$ small enough its
level set $E^{\ge h} (\varphi_\treegraph^1)$ percolates.

Let $\mathcal C_\rot^{h,p}(t)$ be the connected component of
$\{x\in V( \treegraph) : \varphi_\treegraph(x) \ge h$,
  ${\varphi^1_\treegraph(x) \ge h}$, ${\iota (x) =1}\}$ containing $\rot$,
and abbreviate $\eta (h,p) \coloneqq \eta (h,p,-\infty)$.

\begin{proposition}
  \label{pro:pruned}
  For every  $\delta \in (0,1)$, $h<h_\star$, and $p \in [0,1]$ such that
  $(h,p,-\infty)\in \mathcal S^0$,
  \begin{equation}
    \label{eqn:prunedb}
    \lim_{\substack{k\to \infty\\t\to 0}}
    P \Big(\abs{\mathcal C_\rot^{h,p}(t)\cap
        S_\treegraph(\rot,k)}
      \ge \big(p(1-\delta )\lambda_h\big)^{k}\Big)  = \eta (h,p).
  \end{equation}
  (In the limit we allow $k\to\infty$ and then $t\to 0$, or
  $t\to 0, k\to\infty$ together.)
\end{proposition}

\begin{proof}
  Since $\mathcal C_\rot^{h,p}(t)\subset \mathcal C_\rot^{h,p, -\infty}$,
  the left-hand side of \eqref{eqn:prunedb} is bounded from above by
  $\lim_{k\to\infty} P(\abs{\mathcal C_\rot^{h,p,-\infty}
      \cap S_\treegraph(\rot,k)} \ge 1)
  = P(\abs{\mathcal C_\rot^{h,p,-\infty}}=\infty) = \eta (h,p)$,
  by \eqref{eqn:etahpg}, yielding the upper bound in \eqref{eqn:prunedb}.

  To prove the lower bound, we will use  Proposition~\ref{pro:robust}(b) with
  $h'\in (h,h_\star)$ and ${\gamma >-\infty}$, and show that when $t>0$ is
  small enough, then subtracting $\varphi_\treegraph^2$ ``does not destroy
  $\mathcal C_\rot^{h',p,\gamma }$ too much''. To this end, recall
  from~\eqref{eqn:phionetwo} that $\varphi^2_\treegraph$ is an
  i.i.d.~field. However, it is not independent of $\varphi_\treegraph$,
  so we need to compute its conditional distribution given
  $\varphi_\treegraph$.

  \begin{lemma}
    \label{lem:condZ}
    Conditionally on $\varphi_\treegraph$, $\varphi_\treegraph^2$ is a Gaussian
    field determined by
    \begin{align}
      \label{eqn:condexpZ}
      E\big(\varphi^2_\treegraph(x) \mid \sigma (\varphi_\treegraph)\big) &= \tfrac {t^2} {2}
      \Big(\varphi_\treegraph(x) - \tfrac 1d \sum_{z\sim x}
        \varphi_\treegraph (z)\Big),
      \\
      \label{eqn:condcovZ}
      E\big( \varphi^2_\treegraph(x) \varphi^2_\treegraph(y) \mid
        \sigma(\varphi_\treegraph)\big)
      &= \tfrac {t^2}2 \delta_{x,y}- \tfrac {t^4}{4}\big(
      \delta_{x,y} - \tfrac 1 {d} \bbone_{x\sim y}\big).
    \end{align}
  \end{lemma}
  \begin{proof}
    By \eqref{eqn:defgfftd}, \eqref{eqn:ttZvariances},
    and \eqref{eqn:phionetwo},
    the fields $\varphi^2_\treegraph $ and
    $\varphi_\treegraph$ are centred jointly Gaussian fields satisfying
    \begin{equation}
      \label{eqn:phicov}
      \begin{split}
        E (\varphi^2_\treegraph(x) \varphi^2_\treegraph(y))
        &=E (\varphi^2_\treegraph(x) \varphi_\treegraph(y))=
        t^2\delta_{x,y}/2,
        \\
        E(\varphi_\treegraph(x) \varphi_\treegraph(y)) &= g_\treegraph(x,y)
      \end{split}
    \end{equation}
    for every $x,y\in V(\treegraph)$. Denoting $Q$ the transition matrix
    of the usual random walk on $\treegraph$, we observe that
     that for every $x,y\in V(\treegraph)$
    \begin{equation}
      \begin{split}
        \Cov\big(&\varphi_\treegraph^2(x)
          - \tfrac {t^2}{2}((\Id - Q)\varphi_\treegraph)(x),
          \varphi_\treegraph(y)\big)
        \\&  = \tfrac {t^2}2 \Big(\delta_{x,y}
        -\sum_{z\in V(\treegraph)}
        (\Id - Q)(x,z) \Cov\big(\varphi_\treegraph(z),
          \varphi_\treegraph(y)\big)\Big)
        \\&  \refoverseteq{eqn:defgfftd} \tfrac {t^2}2 \Big(\delta_{x,y}
        -\sum_{z\in V(\treegraph)}
        (\Id - Q)(x,z) g_\treegraph(z,y)\Big) = 0,
      \end{split}
    \end{equation}
    where in the last equality we used the well-known identity
    $(\Id - Q)g_{\treegraph} = \Id$ for the Green function.   It follows that the field
    $\psi:=\varphi_\treegraph^2 - \frac{t^2}{2}(\Id - Q)\varphi_\treegraph$ is
    independent of $\sigma (\varphi_\treegraph)$. Hence,
    \begin{equation}
      E\big(\varphi^2_\treegraph \mid \sigma (\varphi_\treegraph)\big)
      = E\big (\psi + \tfrac{t^2}{2}(\Id - Q)\varphi_\treegraph
        \mid \sigma (\varphi_\treegraph)\big)
      =\tfrac{t^2}{2}(\Id - Q)\varphi_\treegraph,
    \end{equation}
    from which \eqref{eqn:condexpZ} follows.

    The conditional covariance of
    $\varphi_\treegraph^2$  agrees with the covariance of $\psi$ (see
      e.g.~\cite[Corollary~1.10]{Le-Gal16}), which is
    \begin{equation}
    \begin{split}
      E (\psi(x)\psi(y)) &=
        E(\varphi_\treegraph^2(x)\varphi_\treegraph^2(y))
        \\&-
        \tfrac{t^2}2 \sum_{z\in V(\treegraph)} (\Id-Q)(y,z)
        E(\varphi_\treegraph^2(x)\varphi_\treegraph(z))
        \\&-
        \tfrac{t^2}2 \sum_{z\in V(\treegraph)} (\Id-Q)(x,z)
        E(\varphi_\treegraph^2(y)\varphi_\treegraph(z))
        \\&+ \tfrac {t^4}4 \sum_{z,z'\in
          V(\treegraph)}(\Id-Q)(x,z)(\Id-Q)(y,z')
        E(\varphi_\treegraph(z) \varphi_\treegraph(z')).
        \end{split}
    \end{equation}
    Statement \eqref{eqn:condcovZ} then follows by inserting the values
    of the expectations from~\eqref{eqn:phicov}, and by applying once more the
    above identity for the Green function.
  \end{proof}

  We continue with the proof of Proposition~\ref{pro:pruned}.
  Consider an arbitrary $h'>h$.
  If $x\in \mathcal C_\rot^{h',p,\gamma }\setminus \{\rot\}$, then
  ${\varphi_\treegraph(x) \ge h'}$ and $\varphi_{\treegraph}(\anc(x))\ge h'$.
  Therefore, by the robustness condition \eqref{eqn:robust}, for
  $x\in \mathcal C_\rot^{h',p,\gamma }\setminus \{\rot\}$,
  \begin{equation}
    \begin{split}
      E\big(\varphi_\treegraph^1(x) \mid \sigma (\varphi_\treegraph)\big)
      &= E\big(\varphi_\treegraph(x) - \varphi_\treegraph^2(x) \mid
        \sigma(\varphi_\treegraph)\big)
      \\&\refoverseteq{eqn:condexpZ}
      \varphi_\treegraph (x) - \tfrac {t^2}2 \Big(\varphi_\treegraph (x)-\tfrac 1d
        \sum_{z\sim x} \varphi_\treegraph(z)\Big)
      \\&= (1-\tfrac {t^2}2)\varphi_\treegraph(x) + \tfrac {t^2}{2d}
      \varphi_\treegraph(\anc(x)) + \tfrac {t^2}{2d} \sum_{z\in \children(x)}
      \varphi_\treegraph(z)
      \\[-1mm]&\refoversetgeq{eqn:robust} h' +t^2(\tfrac \gamma{2d}-\tfrac 12  +\tfrac {h'}{2d}).
    \end{split}
  \end{equation}
  By \eqref{eqn:condcovZ},
  $\Var \big( \varphi_\treegraph^1(x) \mid \sigma(\varphi_\treegraph)\big) \le c t^2$.
  For $t$ small, $h'+t^2(\tfrac \gamma{2d}-\tfrac 12  -\tfrac {h'}{2d})> h$ and thus
  \begin{equation}
    \label{eqn:phioneprob}
    \lim_{t\downarrow 0}
    P\big(\varphi^1_\treegraph(x) \ge h \mid \sigma(\varphi_\treegraph)\big) =1,
    \qquad \text{uniformly for }x\in
    \mathcal C_\rot^{h',p,\gamma }\setminus \{\rot\}.
  \end{equation}
  A similar computation implies that \eqref{eqn:phioneprob} holds for
  $x=\rot$ as well. In addition, by \eqref{eqn:condcovZ}, conditionally
  on $\varphi_\treegraph$, the random variables $\varphi^1_\treegraph(x)$,
  $\varphi^1_\treegraph(y)$ are independent whenever ${d_\treegraph(x,y)\ge 2}$.
  By the domination argument of \cite{LSS97}, the family
  $(\bbone_{[h,\infty)}(\varphi^1_\treegraph(x)): {x\in \mathcal C_\rot^{h',p,\gamma }})$
  dominates (conditionally on $\varphi_\treegraph$) an independent
  Bernoulli percolation on $\mathcal C_\rot^{h',p,\gamma }$ with
  parameter $g(t)$ and $g(t ) \uparrow 1$ as $t\downarrow 0$. As
  consequence:
  \begin{equation}
    \label{eqn:domination}
    \parbox{10cm}{For every $\varepsilon >0$, $h'>h$ and $\gamma\in
      \mathbb R $ there is $t_0=t_0(h,h',\gamma ,\varepsilon )$ such that
       $\mathcal C_0^{h,p}(t)$ dominates
       $\mathcal C_0^{h',p(1-\varepsilon ),\gamma }$ for all $t<t_0$.}
  \end{equation}

  We now fix $\delta >0$ and $\varepsilon < \delta /4$. By the continuity of
  $\lambda_{h}^{p,\gamma }$ proved in Proposition~\ref{pro:robust}(c),
  there is a neighbourhood $\mathcal U_{\delta  }\subset \mathcal S^0$ of $(h,p,-\infty)$ such that ,
  \begin{equation}
    \lambda_{h'}^{p(1-\varepsilon ),\gamma }
    = p(1-\varepsilon ) \lambda_{h'}^{1,\gamma }
    \ge p \big(1-\tfrac \delta 2\big) \lambda_h
    \qquad
    \text{for every ${(h',p,\gamma )\in \mathcal U_{\delta }}$.}
  \end{equation}
  In particular,
  $(p(1-\delta ) \lambda_h)^k \le (\lambda_{h'}^{p(1-\varepsilon ),\gamma })^k / k^2$
  for $k\ge k_0(\delta )$. Hence, by \eqref{eqn:domination}, for such $k$,
  for every $h'>h$ and $\gamma $ such that
  $(h',p,\gamma )\in \mathcal U_\delta $,
  and for every $t<t_0(h,h',\gamma ,\varepsilon )$, the probability in
  \eqref{eqn:prunedb} satisfies
  \begin{equation}
    \begin{split}
      P \big(&\abs{\mathcal C_\rot^{h,p}(t)\cap
          S_\treegraph(\rot,k)}
        \ge \big(p(1-\delta )\lambda_h\big)^{k}\big)
      \\&\ge P \big(\abs{\mathcal C_\rot^{h',p(1-\varepsilon ),\gamma  }\cap
          S_\treegraph(\rot,k)}
        \ge (\lambda_{h'}^{p(1-\varepsilon ),\gamma })^k/k^2\big).
    \end{split}
  \end{equation}
  Observe that the probability on the right-hand side is independent of $t$.
  Therefore, by Proposition~\ref{pro:robust}(b),
  for every $\varepsilon < \delta /4$, $h'\in (h,h_\star)$ and
  $\gamma $ such that
  $(h',p,\gamma )\in \mathcal U_\delta $,
  \begin{equation}
    \liminf_{\substack{k\to \infty\\t\to 0}}
    P \big(\abs{\mathcal C_\rot^{h,p }(t)\cap
        S_\treegraph(\rot,k)}
      \ge (p(1-\delta )\lambda_h)^k \big) \ge \eta(h',p(1-\varepsilon ),\gamma ),
  \end{equation}
  where we can take $k\to\infty$ and then $t\to 0$, or $k\to\infty$,
  $t\to 0$ simultaneously.
  Since $\varepsilon >0$ can be taken arbitrarily close to $0$ and
  $(h',\gamma )$ close to $(h,-\infty)$, the lower bound for \eqref{eqn:prunedb}
  follows using the continuity of $\eta $ from Proposition~\ref{pro:robust}(c).
\end{proof}

\section{Many mesoscopic components for the pruned field on finite graphs} 
\label{sec:mesoscopic}

As a corollary of Proposition~\ref{pro:pruned} and the coupling stated in
Proposition~\ref{pro:coupling}, we now prove
the existence of many mesoscopic components for the level set of the field
$\Psi_\finitegraph^1$.

To state this result precisely, we need to introduce an additional notation. Let
$\bar Z^2_0 = (\bar Z^2_0(x), {x\in \vfinitegraph})$ be a copy of $Z^2_0$
which is independent of $Z$, $Z_0^1$ and $Z_0^2$,
and set (cf.~\eqref{eqn:Psionetwo})
\begin{equation}
  \label{eqn:barPsi}
  \bar \Psi^2_\finitegraph(x) \coloneqq t
  \Big(\bar Z^2_0(x) - \frac 1{N_n} \sum_{y\in \finitegraph} \bar Z^2_0(y)\Big).
\end{equation}
The field $\bar \Psi^2_\finitegraph$ has the same law as
$\Psi^2_\finitegraph$ and thus
$\bar \Psi_\finitegraph \coloneqq \Psi^1_\finitegraph + \bar \Psi^2_\finitegraph$
has the same law as $\Psi_\finitegraph$, that is it is a zero-average
Gaussian free field on $\finitegraph$.
For $p>1/2$ we define $L = L(p)<0$ by
\begin{equation}
  \label{eqn:L}
  P(\bar Z_0^2(x) \ge L) = p.
\end{equation}
We set
\begin{equation}
  \label{eqn:barvn}
  \bvfinitegraph \coloneqq \{x\in \vfinitegraph: {\bar Z^2_0(x) \ge L} \},
\end{equation}
and use $\bfinitegraph$ to denote the subgraph of $\finitegraph$ induced
by $\bvfinitegraph$.
Finally, for $x\in \bvfinitegraph$, let $\mathcal C_x^{h}(t)$
be the connected component component of the set
$E^{\ge h}(\Psi^1_\finitegraph) \cap E^{\ge h}(\Psi_\finitegraph)$ in~
$\bfinitegraph$.

To understand the reason for this notation, note that eventually,
in Section~\ref{sec:sprinkling}, we will use $\bar \Psi_\finitegraph$, and
not $\Psi_\finitegraph$, to show that the supercritical level set has a
giant component. In particular, we will use the field
$\bar \Psi_\finitegraph^2$ for the sprinkling. At the sites where this
field is very small, it can potentially destroy the connected components
of the level set. To avoid this, we will restrict to $\bfinitegraph$ in our
sprinkling construction. It is also useful to compare the definition of
$\mathcal C_x^h(t)$ with the definition of $\mathcal C_\rot^{h,p}(t)$ in
Section~\ref{sec:pruned}, in particular note that the role of the
percolation $\iota $ is taken by the subgraph $\bfinitegraph$.

\begin{proposition}
  \label{pro:mesoscopic}
  Let $h<h_\star$ and let $p$ be such that $(h,p,-\infty)\in \mathcal S^0$.
  Then there exists $c_h \in (0,1)$  such that for any
  $\delta >0$ and any sequence $t_n \downarrow 0 $,
  \begin{equation}
    \label{eqn:mesoscopic}
    \lim_{n\to\infty} P\Big(\sum_{x\in \vfinitegraph}
      \bbone_{\{\abs {\mathcal C_x^{h}(t_n)}\ge N_n^{c_h}\}}
      \ge (1-\delta )\eta (h,p)N_n\Big) = 1.
  \end{equation}
\end{proposition}

\begin{proof}
  The proof follows the steps of Section~5 of \cite{AC20b} and is an
  application of the second moment method. Some simplifications,
  compared to \cite{AC20b}, are due to the fact that our
  Proposition~\ref{pro:coupling} uses two independent copies of
  $\varphi_\treegraph$, so we do not need to use the decoupling inequalities
  for $\varphi_\treegraph$ as in \cite{AC20b}.

  Let $r_n = c_1 \log N_n$ with $c_1>0$, and set
  \begin{equation}
    \begin{split}
      W_n&\coloneqq
      \{x\in \vfinitegraph: x \text{ is $2r_n$-treelike}\},
      \\
      \widetilde W_n &\coloneqq \big\{(x,x')\in W_n\times W_n:
        B_\finitegraph(x,2 r_n) \cap B_\finitegraph(x', 2r_n)  = \emptyset
        \big\}.
    \end{split}
  \end{equation}
  By \cite[(5.6), (5.7)]{AC20b}, it is possible to fix $c_1$ small,
  such that, for some some $c>0$ and for all $n$ large enough,
  \begin{equation}
    \label{eqn:treelikes}
    \abs{W_n }
    \ge N_n (1-N_n^{-c})
    \quad\text{and}\quad
    \abs{\widetilde W_n}
    \ge N_n^2 (1- N_n^{-c}).
  \end{equation}
  We will prove \eqref{eqn:mesoscopic} with
  \begin{equation}
    \label{eqn:deltaprime}
    c_h \coloneqq c_1 \log \big(p \lambda_{h}(1-2\delta ')\big),
  \end{equation}
  where $\delta '>0$ is small enough so that
  $p \lambda_{h}(1-2\delta ')>1$, which is possible since
  $(h,p,-\infty)\in \mathcal S^0$ implies
  $1<\lambda_h^{p,-\infty} = p \lambda_h$.

  Let $\tilde {\mathcal C}_x^h(t_n)\subset \mathcal C_x^h(t_n)$ be the
  connected component of $ \mathcal C_x^h(t_n) \cap B(x,2r_n)$ containing
  $x$, and define events
  \begin{equation}
    \begin{split}
      &A_x^{\finitegraph ,h} \coloneqq\big\{
        \abs{\tilde{\mathcal C}_x^{h}(t_n) \cap S_\finitegraph(x,r_n)}
        \ge N_n^{c_h}\big\}, \qquad \text{for }x\in \vfinitegraph,
      \\&A_\rot^{\treegraph ,h} \coloneqq\big\{
        \abs{\mathcal C_\rot^{h,p}(t_n) \cap S_\treegraph(\rot,r_n)}
        \ge N_n^{c_h}\big\}.
    \end{split}
  \end{equation}
  We now show
  \begin{equation}
    \label{eqn:mesoscopicaa}
    \lim_{n\to\infty} P\Big(\sum_{x\in W_n}
      \bbone_{A_x^{\finitegraph,h}}
      \ge (1-\delta )\eta (h,p)N_n\Big) = 1,
  \end{equation}
  from which \eqref{eqn:mesoscopic} directly follows.

  To show \eqref{eqn:mesoscopicaa}, for every pair
  $x,x'\in \widetilde W_n$, we use the coupling $\mathbb Q_{n}^{x,x'}$
  from Proposition~\ref{pro:coupling} (with $r=r_n$) to couple
  $\Psi_\finitegraph$ with two independent copies of $\varphi_\treegraph$,
  $\varphi_\treegraph'$ of the Gaussian free field on $\treegraph$. By
  Remark~\ref{rem:Zcoupling}, this coupling also couples the underlying
  fields $Z_0$, $\mathtt Z_0$ and $\mathtt Z_0'$ as in
  \eqref{eqn:Z0coupling}. In addition, we write
  $\mathtt Z_0 = \sqrt{1-t_n^2} \mathtt Z_0^1 + t_n \mathtt Z_0^2$ and
  assume that $Z_0^2(y) = \mathtt Z_0^2(\rho_{x,2r_n}(y))$ for every
  $y\in B_{\finitegraph}(x,2r_n)$, where $\mathtt Z_0^1$, $\mathtt Z_0^2$
  are independent copies of $\mathtt Z_0$. We also use analogous
  statements for $\mathtt Z'$, $\mathtt Z_0^{\prime 2}$ in the ball
  $B_{\finitegraph}(x',2r_n)$.  We also couple the site percolation
  $\iota$ (introduced in the paragraph above \eqref{eqn:robust}) and its
  independent copy $\iota'$ with the field $\bar Z_0^2$ so that
  \begin{equation}
    \begin{split}
      \label{eqn:perccoupling}
      &\iota (\rho_{x,2r_n}(y)) = \bbone_{[L,\infty)}(\bar Z^2_0(y)), \qquad y\in
      B_\finitegraph(x,2r_n),
      \\&\iota' (\rho_{x',2r_n}(y)) = \bbone_{[L,\infty)}(\bar Z^2_0(y)), \qquad y\in
      B_\finitegraph(x',2r_n),
    \end{split}
  \end{equation}
  which is always possible due to the choice \eqref{eqn:L} of $L$ and since
  $B_\finitegraph(x,2r_n)$ and $B_\finitegraph(x',2r_n)$ are disjoint.
  Note that \eqref{eqn:perccoupling} implies
  \begin{equation}
    \rho_{x,2r_n}(\bvfinitegraph \cap B(x,2r_n))
    = \{y\in B_\treegraph(\rot,2r_n): \iota(y) =1\}.
  \end{equation}

  We now fix $\varepsilon  >0$ arbitrary but small enough so that
  \begin{equation}
    \label{eqn:epsfix}
    \lambda_{h+\varepsilon }>(1-\delta ')\lambda_h
    \qquad \text{and}\qquad
    (h+\varepsilon ,p,-\infty)\in \mathcal S_0.
  \end{equation}
  where $\delta '$ was fixed in \eqref{eqn:deltaprime}.
  When all coupling equalities from the last paragraph hold and when the coupling
  $\mathbb Q_n^{x,x'}$ succeeds, that is the complement of the event on
  the left-hand side of \eqref{eqn:coupling} occurs, then it follows from
  the definitions of the components $\tilde{\mathcal C}_x^h(t_n)$ and
  $\mathcal C_\rot^{h,p}(t_n)$ that
  $A_\rot^{\treegraph,h-\varepsilon} \supset A_x^{\finitegraph,h}
  \supset A_\rot^{\treegraph, h+ \varepsilon }$,
  and similarly for~$x'$, replacing
  $A_\rot^{\treegraph, h\pm \varepsilon }$ by their independent copies
  defined in terms of the field $\varphi_\treegraph'$. Hence,  for
  $x\in W_n$,
  \begin{equation}
    \label{eqn:alb}
    P(A_x^{\finitegraph,h})\ge P(A_\rot^{\treegraph,h+\varepsilon }) -
    e(n,\varepsilon ),
  \end{equation}
  where $e(n,\varepsilon )$ is the probability that the coupling fails. By
  Proposition~\ref{pro:coupling}, $e(n,\varepsilon )$ is
  bounded by
  right-hand side of \eqref{eqn:coupling} with $r=r_n$, in particular
  $0\le e(n,\varepsilon ) \le c(\varepsilon ) N_n^{-k}$ for any $k\in \mathbb N$.

  By Proposition~\ref{pro:pruned} (applied with $h+\varepsilon$ instead of $h$,
    and $\delta '$ instead of $\delta $), using also~\eqref{eqn:epsfix},
  we obtain that
  \begin{equation}
    \label{eqn:liminf}
    \liminf_{n\to\infty} P(A_\rot^{\treegraph,h+\varepsilon })
    \ge \eta(h+\varepsilon ,p).
  \end{equation}
  Hence, by \eqref{eqn:alb}, also
  $\liminf_{n\to\infty}P(A_x^{\finitegraph,h}) \ge \eta (h+\varepsilon ,p)$
  for every $x\in W_n$. As consequence, since $\varepsilon >0$ is
  arbitrary, using \eqref{eqn:treelikes} and the continuity of $\eta $
  from Proposition~\ref{pro:robust}(c),
  \begin{equation}
    \label{eqn:firstmoment}
    \liminf_{n\to\infty} \frac 1 {N_n}
    E\Big(\sum_{x\in W_n} \bbone_{A_x^{\finitegraph,h}}\Big)
    \ge \eta (h ,p).
  \end{equation}

  We now compute the variance of the sum in the last display.
  Expanding it, and then using the coupling
  $\mathbb Q_{n}^{x,x'}$ again,
  \begin{equation}
    \label{eqn:aaa}
    \begin{split}
      &\Var \Big( \sum_{x\in W_n}
        \bbone_{A_x^{\finitegraph,h}} \Big)
      = \sum_{x,x' \in W_n}
      \Big(P \big(A_x^{\finitegraph,h}\cap A_{x'}^{\finitegraph,h}\big) -
        P (A_x^{\finitegraph,h})
        P (A_{x'}^{\finitegraph,h})\Big)  .
      \\&\le \abs{(W_n \times W_n)\setminus \widetilde W_n }
      + \sum_{(x,x') \in \widetilde W_n}\Big(
        P(A_\rot^{\treegraph,h-\varepsilon })^2-
        P(A_\rot^{\treegraph,h+\varepsilon })^2
        \Big) + e(n,\varepsilon ).
    \end{split}
  \end{equation}
  By definition of $A_\rot^{\treegraph,h}$, using also that
  $\mathcal C_\rot^{h-\varepsilon ,p}(t_n)
  \subset \mathcal C_\rot^{h-\varepsilon ,p,-\infty}$,
  \begin{equation}
    \begin{split}
      \label{eqn:limsup}
      \limsup_{n\to\infty} P(A_\rot^{\treegraph,h-\varepsilon })
      &= \limsup_{n\to\infty}
      P\big(\abs{\mathcal C_\rot^{h-\varepsilon ,p}(t_n) \cap S_\treegraph(\rot,r_n)}
        \ge N_n^{c_h}\big)
      \\&\le \limsup_{n\to\infty}
      P\big(\abs{\mathcal C_\rot^{h-\varepsilon ,p,-\infty} \cap S_\treegraph(\rot,r_n)}
        \ge 1\big)
      \\& =P(\abs{\mathcal C_\rot^{h-\varepsilon ,p,-\infty}} = \infty) =
      \eta (h-\varepsilon ,p).
    \end{split}
  \end{equation}
  Inequalities \eqref{eqn:treelikes}, \eqref{eqn:liminf}, \eqref{eqn:aaa} and \eqref{eqn:limsup}
  together imply  that
  \begin{equation}
    \begin{split}
      \label{eqn:secondmoment}
      \limsup_{n\to\infty}\frac1{N_n^2} \Var &\Big(
        \sum_{x\in W_n}
        \bbone_{A_{x}^{\finitegraph,h}} \Big)
      \leq  \eta(h-\varepsilon,p)^2 -
      \eta(h+\varepsilon,p )^2.
    \end{split}
  \end{equation}
  The right-hand side of this inequality can be made arbitrary small by
  taking $\varepsilon \downarrow 0$, using
  the continuity of $\eta $.
  Statement \eqref{eqn:mesoscopicaa} then follows from
  \eqref{eqn:firstmoment} and \eqref{eqn:secondmoment} by applying
  Chebyshev inequality.
\end{proof}

We finish this section by a simple lemma which gives a lower bound on the
number of vertices that are contained in small components of the
(non-pruned) field $\Psi_\finitegraph$. This lower bound will be used to
show the upper bound on $\abs{\mathcal C_\Max^{\finitegraph,h}}$ in the
proof of Theorem~\ref{thm:main}. In its statement we use
$\mathcal C_x^{\finitegraph,h}$ to denote connected component of
$E^{\ge h}(\Psi_\finitegraph)$ containing~$x\in \vfinitegraph$.

\begin{lemma}
  \label{lem:smallcomponents}
  Let $\mathcal H_n \coloneqq \{x\in \vfinitegraph:
    \mathcal C_x^{\finitegraph,h}\subset B_\finitegraph(x,r_n/2)\}$ with
  $r_n = c_1 \log N_n$ as in the last proof.  Then for
  every $h<h_\star$ and $\delta >0$
  \begin{equation}
    \lim_{n\to\infty} P (\abs{\mathcal H_n} > (1-\eta (h) -\delta ) N_n) = 1.
  \end{equation}
\end{lemma}

\begin{proof}
  The proof is very similar to the previous one. Due to
  \eqref{eqn:treelikes} it is sufficient to prove the claim for
  $\abs{\mathcal H_n \cap W_n}$ instead of $\abs{\mathcal H_n}$.  For
  $x\in W_n$ define the events
  $A_x^{\finitegraph,h} \coloneqq \{\mathcal C_x^{\finitegraph,h} \subset B_\finitegraph(x,r_n/2)\}$,
  $A_\rot^{\treegraph,h} \coloneqq \{\mathcal C_\rot^{h} \subset B_\treegraph(\rot,r_n/2)\}$.
  Using Proposition~\ref{pro:coupling}, we can couple those events so
  that
  $A_\rot^{\treegraph,h-\varepsilon }\subset A_x^{\finitegraph,h}\subset A_\rot^{\treegraph,h+\varepsilon }$.
  By \eqref{eqn:defeta},
  $\lim_{n\to\infty} P(A_x^{\treegraph,h}) = 1-\eta(h)$. Using the same
  first and second moment method arguments as in the previous proof, the
  lemma easily follows.
\end{proof}

\section{Expansion properties of reduced graphs}
\label{sec:expansion}

Before going to the final sprinkling step, we need a little lemma that
show that particular subgraphs of $\finitegraph$ still have good expansion
properties. To this end recall from \eqref{eqn:barvn} the definition of
the subgraph $\bfinitegraph$. For  $K\in \mathbb R$, let
\begin{equation}
  \hvfinitegraph =  \{x\in \vfinitegraph:
    \bar Z^2_0(x)\ge L, \Psi^1_\finitegraph(x) \ge K\}\subset
  \bvfinitegraph,
\end{equation}
and let $\hfinitegraph$ be the subgraph of $\finitegraph$ induced by
$\hvfinitegraph$. The additional condition $\Psi^1_\finitegraph(x) \ge K$ will later
ensure that, in the sprinkling step, the sites in $\hvfinitegraph$ have a reasonable chance to be
in the level set $E^{\ge h} (\bar \Psi_\finitegraph)$ of the field
$\bar \Psi_\finitegraph$ (defined under \eqref{eqn:barPsi}).
We will always assume that $K\le h$, so that the mesoscopic
connected components $\mathcal C_x^{h}(t_n)$ (in $\bfinitegraph$ as considered in
Proposition~\ref{pro:mesoscopic}) are also connected components in
$\hfinitegraph$.

We now show that $\hfinitegraph$ has good expansion properties, at least
when we only consider its large subsets. Recall from
\eqref{eqn:expansion} that $\beta '$ is the lower bound
on the isoperimetric constants of $\finitegraph$.

\begin{lemma}
  \label{lem:reducedexpansion}
  For every $\delta >0 $, there exist  $K_0 = K_0(\delta )$ and
  $L_0 = L_0(\delta )$ such that for every
  $K<K_0$ and $L<L_0$
  \begin{equation}
    \label{eqn:reducedexpansion}
    P\Big( \inf_{A\subset \mathcal V( \hfinitegraph): \delta  N_n \le \abs
        A\le  N_n/2}
      \frac{\abs {\partial_\hfinitegraph A}}{\abs A} \ge \frac{\beta '}2\Big) \ge 1- N^{-\varepsilon }
  \end{equation}
  with $\varepsilon>0 $ independent of $\delta $.
\end{lemma}

\begin{proof}
  We show that for $K,L$ sufficiently negative,
  $B_1(n) = \{x\in \vfinitegraph: \bar Z_0^2(x) < L\}$ and
  $B_2(n) = \{x\in \vfinitegraph: \Psi^1_\finitegraph(x) < K\}$ satisfy
  \begin{equation}
    \label{eqn:expaa}
    P\big(\abs{B_1(n)}+ \abs{B_2(n)} \le \beta '\delta N_n/2\big)\ge
    1-N^{-\varepsilon }.
  \end{equation}
  The claim of the lemma then follows from
  \eqref{eqn:expansion}. Indeed, on the event in \eqref{eqn:expaa}, for $A$ as in
  \eqref{eqn:reducedexpansion},
  \begin{equation}
    \abs{\partial_\hfinitegraph A } \ge  \abs{\partial_\finitegraph A} -
    (\abs{B_1(n)}+\abs{B_2(n)}) \ge \beta ' \abs A - \beta ' \delta N_n /2
    \ge \beta ' \abs A /2.
  \end{equation}

  To prove \eqref{eqn:expaa}, observe first that $\abs{B_1(n)}$ is a
  binomial random variable with parameters~$N_n$ and
  $p=P(\bar Z^2_0(x) < L)$. Hence, by taking $L_0$ depending on $\delta $ sufficiently small,
  we obtain by the
  standard large deviation estimates that
  $P(\abs{B_1(n)}\ge \beta'\delta N_n/4) \le e^{-c N_n}$ for all $L<L_0$.

  For $B_2(n)$, we use the second moment method again. Observe first that by
  \eqref{eqn:Psionetwo},
  \begin{equation}
    \begin{split}
      \Cov& (\Psi_\finitegraph^1(x), \Psi_\finitegraph^1(y)) =
      \Cov (\Psi_\finitegraph(x), \Psi_\finitegraph(y))
      -\Cov (\Psi^2_\finitegraph(x), \Psi^2_\finitegraph(y))
      \\&
      \refoverseteq{eqn:generalzerogff} G_\finitegraph(x,y) - t_n^2 \Cov (\xi_0^2(x), \xi_0^2(y))
      \refoverseteq{eqn:covxi}   G_\finitegraph(x,y) - \tfrac {t_n^2}2(\delta_{x,y} +
        \tfrac 1 {N_n}).
    \end{split}
  \end{equation}
  In particular, using the estimate \eqref{eqn:greenfunctionest} on
  $G_\finitegraph$, since $t_n\to0$,
  $\sigma^2_x  := \Var(\Psi_\finitegraph^1(x)) = G_{\finitegraph}(x,x) + O(t_n^2)\in (c,c')$
  for some
  $0<c<c'<\infty$, and if  $x\neq y$,
  for some $\varepsilon \in (0,1)$,
  \begin{equation}
    \label{eqn:covpsionebound}
    \Cov (\Psi_\finitegraph^1(x), \Psi_\finitegraph^1(y))
    \le C ({d-1})^{-d_\finitegraph(x,y)} + N_n^{-\varepsilon }.
  \end{equation}
  As consequence, we can fix $K_0$ small enough so that
  \begin{equation}
    \label{eqn:btwoexp}
    E(\abs {B_2(n)})
    = \sum_{x\in \vfinitegraph} P(\Psi^1_\finitegraph(x) < K)
    \le \sum_{x\in \vfinitegraph} e^{-K^2/(2c)}
    \le \beta ' \delta N_n  /8
  \end{equation}
  for every $K\le K_0$.
  By the normal comparison lemma, see
  e.g.~\cite[Theorem~4.2.1]{LLR83}, for $x\neq y \in \vfinitegraph$, we
  then obtain
  \begin{equation}
    \begin{split}
      &P(\Psi_\finitegraph^1(x) \le K , \Psi_\finitegraph^1(y) \le K  )
      -P(\Psi_\finitegraph^1(x) \le K) P(\Psi_\finitegraph^1(y) \le K  )
      \\&\quad\le C (\Cov(\Psi_\finitegraph^1(x)/\sigma_x
          ,\Psi_\finitegraph^1(y)/\sigma_y)\vee 0)
      \le C (\Cov(\Psi_\finitegraph^1(x)
          ,\Psi_\finitegraph^1(y))\vee 0).
    \end{split}
  \end{equation}
  As consequence, using also the fact that diameter of $\finitegraph$ is
  smaller than $C\log N_n$ (see e.g.~\cite[Proposition~3.1.5]{Kow19}) and
  that  $\abs{S_\finitegraph(x,r)}\le d (d-1)^{r-1}$, we obtain
  \begin{equation}
    \begin{split}
      \label{eqn:btwovar}
      \Var &(\abs{B_2(n)}) \\&= \sum_{x,y\in \vfinitegraph }
      P(\Psi^1_\finitegraph(x)\le K, \Psi^1_\finitegraph(y)\le K)
      -
      P(\Psi^1_\finitegraph(x)\le K)P( \Psi^1_\finitegraph(y)\le K)
      \\&\le
      N_n + C\sum_{x\in \vfinitegraph} \sum_{r=1}^{C\log N_n}
      \sum_{y\in \vfinitegraph: d_\finitegraph(x,y)=r}
      (\Cov(\Psi_\finitegraph^1(x) ,\Psi_\finitegraph^1(y))\vee 0)
      \\&\refoversetleq{eqn:covpsionebound}
       N_n + C\sum_{x\in \vfinitegraph} \sum_{r=1}^{C\log N_n}
      \sum_{y\in \vfinitegraph: d_\finitegraph(x,y)=r}
      ((d-1)^{-r} + N_n^{-\varepsilon })
      \le N_n^{2-\varepsilon }.
    \end{split}
  \end{equation}
  By Chebyshev inequality, using \eqref{eqn:btwoexp}, \eqref{eqn:btwovar},
  $P(\abs{B_2(n)}\ge \beta '\delta N_n/4)\le N_n^{-\varepsilon }$.

  Combining the conclusions of the last two paragraphs then implies
  \eqref{eqn:expaa} and completes the proof.
\end{proof}

\section{Sprinkling / Proof of Theorem~\ref{thm:main}} 
\label{sec:sprinkling}

With all preparations of the previous sections, the sprinkling
construction is relatively straightforward and follows the steps of
\cite{ABS04}.

We start by showing that the field $\bar \Psi_\finitegraph$
defined under \eqref{eqn:barPsi} contains a giant component of size at
least $\eta(h)(1-\delta )N_n$, with probability tending to one as
$n\to \infty$. Since $\bar \Psi_\finitegraph$ is a zero-average Gaussian
free field on $\finitegraph$, this will imply the lower bound on
$\abs { \mathcal C_\Max^{\finitegraph,h}}$ for Theorem~\ref{thm:main}.

For $h<h_\star$ and $\delta \in (0,1/8)$ as in the statement of
Theorem~\ref{thm:main}, we fix an arbitrary $h'\in (h,h_\star)$ and set
$\varepsilon \coloneqq h'-h$. We further fix $K$, $L$ small and $p$ close
to $1$ so that $L$, $p$ are linked by \eqref{eqn:L} and
\begin{equation}
  \label{eqn:fixing}
  \begin{aligned}
    &K=h\wedge K_0(\delta \eta (h')/2), \\
    &(h',p,-\infty)\in \mathcal S^0,
  \end{aligned}
  \qquad
  \begin{aligned}
    &L<L_0(\delta \eta(h')/2),
    \\&\eta(h',p) > \eta (h')/2,
  \end{aligned}
\end{equation}
where $K_0(\delta \eta(h')/2)$, $L_0(\delta \eta(h')/2)$ are as in
Lemma~\ref{lem:reducedexpansion},
and the last inequality in \eqref{eqn:fixing} can be satisfied by
Proposition~\ref{pro:robust}(c). We let $t_n\to 0$ slowly so that
\begin{equation}
  \label{eqn:tnfix}
  P(\bar Z^2_0(x) \ge t_n^{-1}(h+ 1 -K) )\ge N_n^{-c_{h'}\beta' \delta
    \eta(h',p)/8},
\end{equation}
where $c_{h'}$ is as in Proposition~\ref{pro:mesoscopic}, and $\beta '$ as
in Lemma~\ref{lem:reducedexpansion}.

Due to Lemma~\ref{lem:reducedexpansion}, using also \eqref{eqn:fixing}, we know that:
\begin{equation}
  \label{eqn:Aone}
  \text{For }
   \mathcal A^1_n
    \coloneqq \bigg\{\inf_{\substack{A\subset V(\hfinitegraph): \\ \delta \eta
          (h')\frac{N_n}2 < |A| < \frac{N_n}2}}
      \frac{
        \abs{\partial_{\hfinitegraph}A}}{\abs A}
      \ge \frac {\beta'}2 \bigg\} \text{ we have } P(\mathcal A^1_n) \ge 1- N_n^{-c}.
\end{equation}
By Gaussian tail estimates, the zero-averaging term in the definition
\eqref{eqn:barPsi} of
$\bar \Psi_\finitegraph^2$ is negligible with high probability:
\begin{equation}
  \label{eqn:zeroaveragepert}
  \text{For }\mathcal A^2_n \coloneqq \Big\{
  \abs[\Big]{N_n^{-1}\sum_{y\in \vfinitegraph} \bar Z^2_0(y)}\le
  \varepsilon \Big\}\text{ we have }
  P( \mathcal A^2_n) \ge 1 - e^{-c N_n}.
\end{equation}
Introducing $m_n:=N_n^{c_{h'}}$ to denote the minimal
size of mesoscopic components and writing
$a_k= (1- k \delta )\eta (h',p)$ for $k\in \{1,2\}$, Proposition~\ref{pro:mesoscopic} implies
that:
\begin{equation}
  \label{eqn:mesoscopichprime}
  \text{For } \mathcal A_n^3 \coloneqq \Big\{\sum_{x\in \vfinitegraph}
      \bbone_{\{\abs {\mathcal C_x^{h'}(t_n)}\ge m_n\}}
      \ge a_1 N_n \Big\} \text{ we have }
    \lim_{n\to\infty} P(\mathcal A_n^3) = 1,
\end{equation}
that is $E^{\ge h'}({\Psi^1_\finitegraph}) \cap \hvfinitegraph$
has many mesoscopic components, with high probability.
Finally, since $\Psi_\finitegraph^1$ is independent of
$\bar Z^2_0$ and the graph $\hfinitegraph$ depends on $\bar Z^2_0$ only
via
$\bbone_{[L,\infty)}(\bar Z^2_0(x))$, it follows that:
\begin{equation}
  \label{eqn:condZbar}
  \parbox{13cm}{Conditionally on
    $\sigma (\Psi_\finitegraph^1,\hfinitegraph)$, the random variables
    $(\bar Z_0^2(x))_{x\in \hvfinitegraph}$ are i.i.d.~distributed as
    $\mathcal N(0,1/2)$ random variable conditioned on being larger than
    $L$.}
\end{equation}
In particular, since $L<0$,
\begin{equation}
  \label{eqn:pn}
  \begin{split}
    p_n &\coloneqq P(\bar Z_0^2(x) \ge t_n^{-1}(h + t_n \varepsilon -K) \mid
      \sigma (\Psi_\finitegraph^1,\hfinitegraph), \{x\in \hvfinitegraph\})
    \\&\refoversetgeq{eqn:tnfix} N_n^{-c_{h'}\beta' \delta
      \eta(h',p)/8}.
  \end{split}
\end{equation}

Assume now that
$\mathcal A_n \coloneqq \mathcal A_n^1 \cap \mathcal A_n^2 \cap \mathcal A_n^3$
occurs. On $\mathcal A_n^3$, we can fix a set of at most $a_1 N_n/m_n$
mesoscopic components of
$E^{\ge h'}(\Psi^1_\finitegraph)\cap \hvfinitegraph$ that together
contain at least $a_1 N_n$ vertices. Any $x$ in those components
satisfies $\bar Z^2_0(x) \ge L$ (by definition of~$\hfinitegraph$) and
thus (on $\mathcal A^2_n$),
\begin{equation}
  \bar \Psi_\finitegraph(x) = \Psi_\finitegraph^1(x) + t_n\Big(\bar Z^2_0(x) -
    N_n^{-1} \sum_{y\in \vfinitegraph} \bar Z^2_0(y)\Big) \ge
  h'+t_n(L-\varepsilon ) \ge h
\end{equation}
for all $n$ large enough. It follows that these fixed mesoscopic
components are contained in $E^{\ge h}(\bar \Psi_{\finitegraph})$. If
$E^{\ge h}(\bar\Psi_\finitegraph)$ has no component of size at least
$a_2 N_n$, then one can split these fixed components into two groups $A$,
$B$, each having at least $\delta \eta( h',p)N_n$ vertices, which are not
connected within $E^{\ge h}(\bar \Psi_\finitegraph)$. There are at most
$2^{a_1N_n/m_n}$ ways to split the fixed mesoscopic components into two
groups. By \eqref{eqn:fixing},
$\delta \eta (h',p) N_n > \delta \eta (h')N_n/2$. Therefore, on
$\mathcal A_n^1$, we can use Menger's theorem to show that there are at
least $\beta '\delta N_{n}\eta (h',p)/2$ pairwise vertex-disjoint paths
from $A$ to $B$ in $\hfinitegraph$. Since $\hfinitegraph$ has at most
$N_n$ vertices, at last half of those paths are of length at most
$4/\beta ' \delta \eta (h',p)$ each. For every $x\in \hvfinitegraph$,
$\Psi_\finitegraph^1(x) \ge K$. Therefore, if
$\bar Z^2_0(x)\ge t_n^{-1}(h + t_n \varepsilon -K)$ and $\mathcal A^2_n$
occurs, then
\begin{equation}
  \bar \Psi_\finitegraph(x) = \Psi_\finitegraph^1(x) + t_n\Big(\bar
    Z_0^2(x) - \frac 1 {N_n}\sum_{x\in \vfinitegraph} \bar
    Z_0^2(x)\Big)\ge h.
\end{equation}
Hence for the groups $A$ and $B$ being disconnected in
$\hfinitegraph \cap E^{\ge h}(\bar \Psi_\finitegraph)$, there must be
at least one vertex with $\bar Z^2_0(x) < t_n^{-1}(h + t_n \varepsilon -K)$ on
every of these paths. Due to \eqref{eqn:condZbar} and \eqref{eqn:pn},
this has probability at most
\begin{equation}
  (1-p_n^{4/\beta ' \delta \eta (h',p)})^{\beta ' \delta N_n \eta (h',p)/4} \le
  \exp(- c(\delta ,h',p) N_n^{1-c_{h'}/2}).
\end{equation}
It follows that the probability that $\mathcal A_n$ occurs and there is
no connected component of $E^{\ge h}(\bar \Psi_\finitegraph)$ of size
at least $a_2 N_n$ (that is there is some partition of the fixed mesoscopic
  components into groups $A$ and $B$ as above that are disconnected form
  each other in $E^{\ge h}(\bar \Psi_\finitegraph)$) is at most
\begin{equation}
   2^{a_1N_n/m_n}
  \exp(- c(\delta ,h',p) N_n^{1-c_{h'}/2})\le \exp\{-c'
    N_n^{1-c_{h'}/2}\},
\end{equation}
which converges to $0$ as $n\to \infty$.

Together with \eqref{eqn:Aone}--\eqref{eqn:mesoscopichprime}, this
implies that with probability tending to one with $n$,
$E^{\ge h}(\bar \Psi_\finitegraph)$ has a connected component of size at
least $a_2 N_n = (1-2\delta )\eta (h',p) N_n$. Taking $h'$ close to $h$,
$p$ close to $1$, using the continuity of $\eta(h,p)$ from
Proposition~\ref{pro:robust}(c), and recalling that
$\bar\Psi_\finitegraph$ has the same distribution as $\Psi_\finitegraph$
then proves the lower bound on $\abs{\mathcal C_\Max^{\finitegraph,h}}$
for our main result~\eqref{eqn:main} of Theorem~\ref{thm:main}.

The upper bounds on $\abs{\mathcal C_\Max^{\finitegraph,h}}$ and
$\abs{\mathcal C_\Sec^{\finitegraph,h}}$ in \eqref{eqn:main} then follow
from Lemma~\ref{lem:smallcomponents} and the lower bound on
$\abs{\mathcal C_\Max^{\finitegraph,h}}$.  This completes the proof of
Theorem~\ref{thm:main}.

\begin{remark}
  \label{rem:assumptions}
  We conclude this paper with a short discussion of the assumptions of
  Theorem~\ref{thm:main}. Assumption~\ref{assumptions}(a) is
  clearly necessary in all our considerations (besides
    Section~\ref{sec:decomposition}).

  Assumption~\ref{assumptions}(c), that is the assumption on the spectral
  gap, is only used to imply the uniform isoperimetric inequality
  \eqref{eqn:expansion}, and also in \eqref{eqn:mchconv}. For our results
  to be true, we only need \eqref{eqn:expansion} to hold for macroscopic
  sets (cf.~proof of Lemma~\ref{lem:reducedexpansion}). Also, the argument
  around \eqref{eqn:mchconv} can be easily adapted if
  $\lambda_\finitegraph\to 0$ sufficiently slowly.

  Assumption~\ref{assumptions}(b) is only used very implicitly in this
  paper, namely to ensure that a majority of vertices of $\finitegraph$
  are $r_n$-treelike with $r_n = c_1 \log N_n$, cf.~\eqref{eqn:treelikes}
  which is proved in \cite[(5.6)]{AC20b} using \cite[Lemma~6.1]{CTW11}.
  In Sections~\ref{sec:mesoscopic}--\ref{sec:sprinkling} of this paper,
  we even do not need that $r_n$ grows so quickly. $r_n = C \log \log N_n$
  for $C$ sufficiently large would be sufficient for our purposes. Hence
  Assumption~\ref{assumptions}(b) can be replaced by: For some $C$
  sufficiently large,
  \begin{equation}
    \abs{\{x\in \vfinitegraph: x \text{ is
          $(C \log \log N_n)$-treelike}\}}\ge N_n(1-o(1)).
  \end{equation}
  For the existence of the giant component (not necessary of size
     ${(1-\delta )\eta (h) N_n}$), the factor $(1-o(1))$ in the last
  inequality could even be replaced by a $c\in (0,1)$.
\end{remark}

\ack

The author wish to thank P.-F.~Rodriguez for useful discussions.


\frenchspacing

\begin{thebibliography}{10}
\providecommand{\arXiv}[1]{{{\href{https://arxiv.org/abs/#1}{\emph{arXiv:#1}}}}}
\providecommand{\url}[1]{\texttt{#1}}
\providecommand{\urlprefix}{Available at: }
\providecommand{\eprint}[2][]{\url{#2}}

\bibitem{Aba19}
Ab{\"a}cherli, A. Local picture and level-set percolation of the {G}aussian
  free field on a large discrete torus. \emph{Stochastic Process. Appl.}
  \textbf{129} (2019), no.~9, 3527--3546.
\urlprefix\url{https://doi.org/10.1016/j.spa.2018.09.017}

\bibitem{AS18}
Ab{\"a}cherli, A.; Sznitman, A.-S. Level-set percolation for the {G}aussian
  free field on a transient tree. \emph{Ann. Inst. H. Poincaré Probab.
  Statist.} \textbf{54} (2018), no.~1, 173--201.
\urlprefix\url{https://doi.org/10.1214/16-AIHP799}

\bibitem{AC20a}
{Ab{\"a}cherli}, A.; {{\v{C}}ern{\'y}}, J. Level-set percolation of the
  {G}aussian free field on regular graphs {I}: Regular trees. \emph{Electron.
  J. Probab.} \textbf{25} (2020), Paper No. 65, 24.
\urlprefix\url{https://doi.org/10.1214/20-ejp468}

\bibitem{AC20b}
Ab\"{a}cherli, A.; \v{C}ern\'{y}, J. Level-set percolation of the {G}aussian
  free field on regular graphs {II}: finite expanders. \emph{Electron. J.
  Probab.} \textbf{25} (2020), Paper No. 130, 39.
\urlprefix\url{https://doi.org/10.1214/20-ejp532}

\bibitem{ABS04}
Alon, N.; Benjamini, I.; Stacey, A. Percolation on finite graphs and
  isoperimetric inequalities. \emph{Ann. Probab.} \textbf{32} (2004), no.~3A,
  1727--1745.
\urlprefix\url{https://doi.org/10.1214/009117904000000414}

\bibitem{AL15}
Anantharaman, N.; L{e M}asson, E. Quantum ergodicity on large regular graphs.
  \emph{Duke Math. J.} \textbf{164} (2015), no.~4, 723--765.
\urlprefix\url{https://doi.org/10.1215/00127094-2881592}

\bibitem{BLM87}
Bricmont, J.; Lebowitz, J.~L.; Maes, C. Percolation in strongly correlated
  systems: the massless {G}aussian field. \emph{J. Statist. Phys.} \textbf{48}
  (1987), no. 5-6, 1249--1268.
\urlprefix\url{https://doi.org/10.1007/BF01009544}

\bibitem{CTW11}
{\v{C}}ern{\'y}, J.; Teixeira, A.; Windisch, D. Giant vacant component left by
  a random walk in a random {$d$}-regular graph. \emph{Ann. Inst. Henri
  Poincar\'e Probab. Stat.} \textbf{47} (2011), no.~4, 929--968.
\urlprefix\url{http://dx.doi.org/10.1214/10-AIHP407}

\bibitem{CN20}
Chiarini, A.; Nitzschner, M. Entropic repulsion for the {G}aussian free field
  conditioned on disconnection by level-sets. \emph{Probab. Theory Related
  Fields} \textbf{177} (2020), no. 1-2, 525--575.
\urlprefix\url{https://doi.org/10.1007/s00440-019-00957-7}

\bibitem{Con21}
Con{chon-Kerjan}, G.: Anatomy of a {G}aussian giant: supercritical level-sets
  of the free field on random regular graphs, Preprint, available at:
  \arXiv{2102.10975}. 2021.

\bibitem{CF13}
Cooper, C.; Frieze, A. Component structure of the vacant set induced by a
  random walk on a random graph. \emph{Random Structures Algorithms}
  \textbf{42} (2013), no.~2, 135--158.
\urlprefix\url{https://doi.org/10.1002/rsa.20402}

\bibitem{DPR18}
Drewitz, A.; Pr\'{e}vost, A.; Rodriguez, P.-F. The sign clusters of the
  massless {G}aussian free field percolate on $\mathbb{Z}^d$, $d\geq3$ (and
  more). \emph{Comm. Math. Phys.} \textbf{362} (2018), no.~2, 513--546.
\urlprefix\url{https://doi.org/10.1007/s00220-018-3209-6}

\bibitem{DPR21}
Drewitz, A.; Prévost, A.; Rodriguez, P.-F.: Critical exponents for a
  percolation model on transient graphs, Preprint, available at:
  \arXiv{2101.05801}. 2021.

\bibitem{DR15}
Drewitz, A.; Rodriguez, P.-F. High-dimensional asymptotics for percolation of
  {G}aussian free field level sets. \emph{Electron. J. Probab.} \textbf{20}
  (2015), 1--39.
\urlprefix\url{https://doi.org/10.1214/EJP.v20-3416}

\bibitem{DGRS20}
Dumini{l-C}opin, H.; Goswami, S.; Rodriguez, P.-F.; Severo, F.: Equality of
  critical parameters for percolation of {G}aussian free field level-sets,
  Preprint, available at: \arXiv{2002.07735}. 2020.

\bibitem{ES10}
Elon, Y.; Smilansky, U. Percolating level sets of the adjacency eigenvectors of
  $d$-regular graphs. \emph{Journal of Physics A: Mathematical and Theoretical}
  \textbf{43} (2010), no.~45, 455\,209.
\urlprefix\url{https://doi.org/10.1088/1751-8113/43/45/455209}

\bibitem{Kow19}
Kowalski, E. \emph{An introduction to expander graphs}, \emph{Cours
  Sp\'{e}cialis\'{e}s [Specialized Courses]}, vol.~26, Soci\'{e}t\'{e}
  Math\'{e}matique de France, Paris, 2019.

\bibitem{KLS20}
Krivelevich, M.; Lubetzky, E.; Sudakov, B. Asymptotics in percolation on
  high-girth expanders. \emph{Random Structures Algorithms} \textbf{56} (2020),
  no.~4, 927--947.
\urlprefix\url{https://doi.org/10.1002/rsa.20903}

\bibitem{Le-Gal16}
Le~Gall, J.-F. \emph{Brownian motion, martingales, and stochastic calculus},
  \emph{Graduate Texts in Mathematics}, vol. 274, Springer [Cham], 2016.
\urlprefix\url{https://doi.org/10.1007/978-3-319-31089-3}

\bibitem{LLR83}
Leadbetter, M.~R.; Lindgren, G.; Rootz{\'e}n, H. \emph{Extremes and related
  properties of random sequences and processes}, Springer Series in Statistics,
  Springer-Verlag, New York, 1983.

\bibitem{LS86}
Lebowitz, J.~L.; Saleur, H. Percolation in strongly correlated systems.
  \emph{Phys. A} \textbf{138} (1986), no. 1-2, 194--205.
\urlprefix\url{https://doi.org/10.1016/0378-4371(86)90180-9}

\bibitem{LPW09}
Levin, D.~A.; Peres, Y.; Wilmer, E.~L. \emph{Markov chains and mixing times},
  American Mathematical Society, Providence, RI, 2009. With a chapter by James
  G. Propp and David B. Wilson.

\bibitem{LSS97}
Liggett, T.~M.; Schonmann, R.~H.; Stacey, A.~M. Domination by product measures.
  \emph{Ann. Probab.} \textbf{25} (1997), no.~1, 71--95.
\urlprefix\url{http://doi.org/10.1214/aop/1024404279}

\bibitem{MS83}
Molchanov, S.~A.; Stepanov, A.~K. Percolation in random fields. {I}.
  \emph{Teoret. Mat. Fiz.} \textbf{55} (1983), no.~2, 246--256.

\bibitem{MR98}
Molloy, M.; Reed, B. The size of the giant component of a random graph with a
  given degree sequence. \emph{Combin. Probab. Comput.} \textbf{7} (1998),
  no.~3, 295--305.
\urlprefix\url{https://doi.org/10.1017/S0963548398003526}

\bibitem{RS80}
Reed, M.; Simon, B. \emph{Methods of modern mathematical physics. {I}},
  Academic Press, Inc., New York, 1980, 2nd ed.

\bibitem{RS13}
Rodriguez, P.-F.; Sznitman, A.-S. Phase transition and level-set percolation
  for the {G}aussian free field. \emph{Comm. Math. Phys.} \textbf{320} (2013),
  no.~2, 571--601.
\urlprefix\url{https://doi.org/10.1007/s00220-012-1649-y}

\bibitem{Szn16}
Sznitman, A.-S. Coupling and an application to level-set percolation of the
  {G}aussian free field. \emph{Electron. J. Probab.} \textbf{21} (2016), 1--26.
\urlprefix\url{https://doi.org/10.1214/16-EJP4563}

\bibitem{Szn19a}
Sznitman, A.-S. On macroscopic holes in some supercritical strongly dependent
  percolation models. \emph{Ann. Probab.} \textbf{47} (2019), no.~4,
  2459--2493.
\urlprefix\url{https://doi.org/10.1214/18-AOP1312}

\end{thebibliography}
\end{document}